%% file: main.tex
\documentclass[11pt,a4paper, reqno, dvipsnames, oneside]{amsart}
\usepackage{amssymb}
\usepackage{geometry}
\usepackage{xcolor}
\usepackage{graphicx}
\usepackage{comment}
\usepackage{enumitem}
\setlist{listparindent=0pt}

\usepackage{tikz}
\usetikzlibrary{cd}
\usetikzlibrary{positioning,calc,arrows.meta}
\usepackage{forest}

\forestset{
  bdot/.style={fill=black},
  gdot/.style={fill=gray, color=gray},
  wdot/.style={fill=white},
}

\DeclareRobustCommand{\circledvertex}{%
  \tikz[baseline=-0.6ex, semithick]{\node[circled] {};}}
\DeclareRobustCommand{\hollowvertex}{%
  \tikz[baseline=-0.6ex, semithick]{\node[hollow] {};}}

\tikzset{
  diagram/.style   = {x=0.4cm, y=0.75cm, semithick}, 
  vertex/.style    = {circle, fill, inner sep=0pt, minimum size=4.5pt},
  hollow/.style    = {circle, draw, fill=white, inner sep=0pt, minimum size=6pt},
  circled/.style   = {circle, draw, inner sep=0pt, minimum size=8.5pt,
                      path picture={\fill (path picture bounding box.center) circle (2.25pt);}},
  bold edge/.style = {line width=1.6pt},
}

\tikzset{
  inline label/.style={font=\footnotesize, inner sep=1.5pt}, 
  inlinediag/.style={
    diagram,
    reset cm,
    x=1.25em, y=1.25em,  
    baseline=-0.5ex,
    every label/.append style={inline label},
  },
}

\NewDocumentCommand{\TypeIinline}{}{%
  \tikz[inlinediag]{
    \node[vertex, label={above:$1/2$}]       (m) at (0,0.5)   {};
    \node[hollow, label={above left:$1/3$}]  (h) at (-1,-0.5) {};
    \node[vertex, label={above right:$2/3$}] (r) at (1,-0.5)  {};
    \draw (m) -- (h);
    \draw[bold edge] (m) -- (r);
  }%
}

\newenvironment{inlinematrix}
  {\bigl(\begin{smallmatrix}}
  {\end{smallmatrix}\bigr)}

\usepackage{mymacros}

\usepackage{hyperref}
\usepackage[capitalize]{cleveref}
\hypersetup{
    colorlinks,
    linkcolor=Violet,
    citecolor=Violet,
    urlcolor=black,
}

\theoremstyle{plain}
\newtheorem{theorem}{Theorem}
\newtheorem{lemma}{Lemma}
\newtheorem{corollary}{Corollary}
\newtheorem{proposition}{Proposition}

\theoremstyle{definition}
\newtheorem{definition}{Definition}

\newtheorem{example}{Example}
\newtheorem{remark}{Remark}

\newtheorem{problem}{Problem}

\theoremstyle{remark}

\newlist{thmparts}{enumerate}{1}
\setlist[thmparts]{label=(\roman*), font=\upshape, leftmargin=*, widest=ii}

\title{On the $d$-invariants of Brieskorn spheres and \texorpdfstring{$\GL(2,\Z)$}{GL(2,Z)}}
\author{Shimal Harichurn, Isabella Khan, Andr\'as N\'emethi, Josef Svoboda}
\date{\today}

\address{Durban, KwaZulu-Natal, South Africa}
\email{sharichurn.research@gmail.com}

\address{School of Mathematics, Georgia Institute of Technology, 686 Cherry Street, Atlanta, GA 30332-0160, USA}
\email{ikhan312@gatech.edu}

\address{Alfr\'ed R\'enyi Institute of Mathematics, 1053 Budapest, Re\'altanoda u. 13-15., Hungary}
\email{nemethi.andras@renyi.hu}

\address{Department of Mathematics, California Institute of Technology, 1200~E.~California Blvd., Pasadena, CA 91125, USA}
\email{josefsvobod@gmail.com}

\begin{document}

\begin{abstract}
    We study the $d$-invariant from Heegaard Floer homology for Brieskorn spheres $\Sigma(p,q,r)$. Using surgeries on two singular fibers, we establish new relations of the $d$-invariants among various Brieskorn spheres. For $p=2$, we show that $d$ never vanishes if $q/r$ lies in certain intervals on the rational line. Altogether, we completely classify Brieskorn spheres $\Sigma(2,q,r)$ with $d=0$. The zero set is expressed using submonoids of ${\rm GL}(2,\Z)$ and the Stern--Brocot tree. As an application, we prove that most Brieskorn spheres $\Sigma(2,q,r)$ do not bound a rational homology $4$-ball.
\end{abstract}

\maketitle

\tableofcontents

\section{Introduction}\label{sec:intro}

Heegaard Floer homology, defined by Ozsv\'ath and Szab\'o in~\cite{OzsvathSzabo2004}, is a powerful invariant of 3-dimensional manifolds. For an integer homology 3-sphere $Y$, Heegaard Floer homology admits an absolute grading, whose minimal value is an even integer $d(Y)$ \cite{OzsvathSzaboGraded}. This $d$-invariant is a rational homology cobordism invariant and as such it plays an important role in the study of smooth 4-manifolds.

A lot of attention was given to the computation of the $d$-invariants for plumbed 3-manifolds \cite{os_floer_plumbed,Nemethi2005,khan2026, nemethi2007}, and in particular Brieskorn spheres $\Sigma(p,q,r)$ \cite{durusoy2004, rustamov2003, Borodzik2012, Can2014, Karakurt_Savk_2020, Karakurt_2022, Tweedy_2013}. These 3-manifolds, labelled by a triple $p,q,r$ of relatively prime integers, arise as links of singularities
\begin{equation}\label{eq:sing}
    x^p + y^q + z^r = 0.
\end{equation}
Despite the number of works on this topic, the $d$-invariant has been explicitly determined only for certain special families of Brieskorn spheres.

\begin{figure}[tbp]
    \centering
    \includegraphics[width=\linewidth]{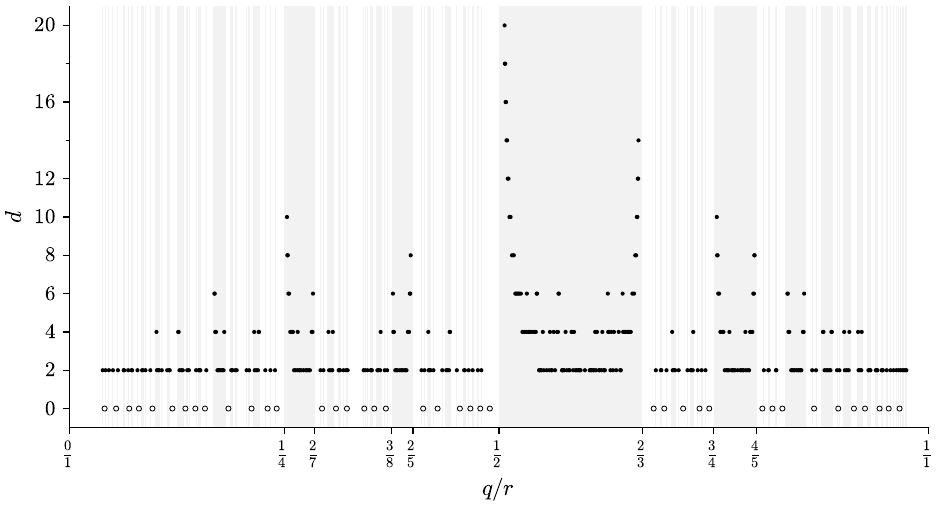}
    \caption{Plot of $d(\Sigma(2, q, r))$ against $q/r \in [0, 1]$ for coprime $2 < q < r < 80$. For visual clarity, only a subset of the $d=0$ set is shown. Shaded regions correspond to \textit{maximal even intervals} studied in the paper.}
    \label{fig:d2qr}
\end{figure}
 
In this paper, we make a significant step towards determining $d$ for all Brieskorn spheres. The key idea is to use surgeries on two singular fibers. Namely, we fix $p$ and view the union of singular fibers $K_q$ and $K_r$ of the Seifert fibration as a two-component link $L$ inside $\Sigma(p,q,r)$. We then perform a surgery on $L$, while ensuring that the result is another Brieskorn sphere $\Sigma(p,q',r')$ with the same (negative) orientation. Although we can do the surgery in two steps (first on $K_q$ and then on $K_r$), the intermediate step does not have to be negatively oriented.

It turns out that many of these surgeries preserve the $d$-invariant. Among those are surgeries on one singular fiber only, studied by Lidman and Tweedy \cite{Lidman_2018}. An example of a new one is\footnote{One can do $(p-1,p+1)$-surgery for $p>2$, but we specialize to $p=2$ for simplicity.} the $(1,3)$-surgery for $p=2$, which transforms $\Sigma(2,q,r)$ to $\Sigma(2,-q+2r,-2q+3r)$. This allows us to transfer the study of the $d$-invariants to the study of submonoids of $\GL(2,\Z)$ and their actions.

We formulate the surgeries in the language of plumbing lattices. This is convenient, as the $d$-invariant is determined by the lattice. However, the surgeries can be used to obtain relations among other invariants, such as Heegaard Floer homology itself, its involutive version \cite{Dai_Manolescu_2019} etc. In that case, surgeries should be formulated in terms of corresponding cobordisms among Brieskorn spheres.

We use the invariance of $d$ under surgeries, together with a non-vanishing result (\cref{thm:nonvanishing_intro} below), to provide a complete description of the $d=0$ set among the spheres $\Sigma(2, q, r)$. Viewing the $d$-invariant as a function of the quotient $q/r \in (0,1)$, our result is best visualized on the Stern--Brocot tree in \cref{fig:tree_intro}. The thick edges correspond to intervals on which $d>0$. The empty dots then correspond to the Brieskorn spheres with vanishing $d$. For example, $3/5$ lies in the interval $(1/2,2/3)$ and hence $d(\Sigma(2,3,5)) = 2 > 0$, while for $\Sigma(2,5,7)$ we have $d=0$. Theorems \ref{th:int1}, \ref{th:int2} and \ref{th:int3} below make this structure rigorous and provide different ways of understanding it.

  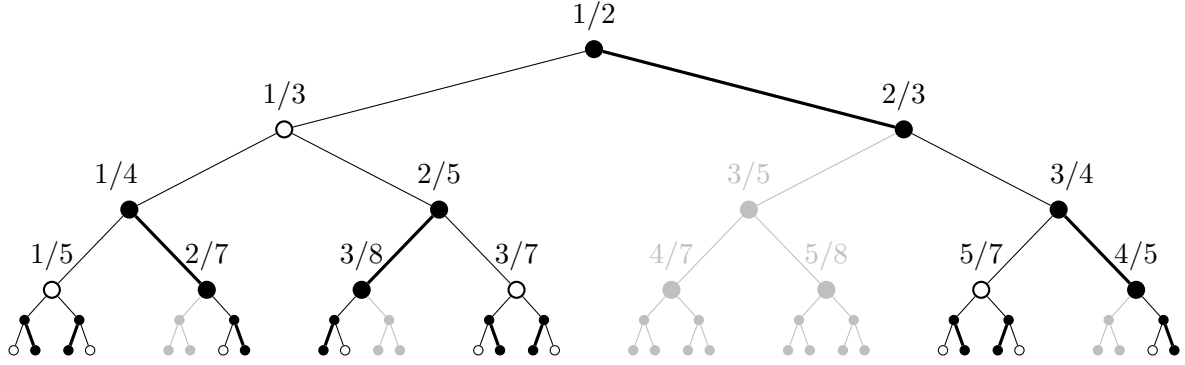
\begin{figure}
\centering
\colorlet{gray}{gray!50}
\begin{forest}
  for tree={
    circle,
    draw,
    thick,
    minimum size=6pt,   
    inner sep=0pt,      
    l sep=4mm,          
  },
  where level<=2{s sep=9mm}{},
  where level=3{s sep=3mm, l sep=1mm}{},     
  where level=4{s sep=1.5mm, l sep=1mm}{},   
  where level>=4{minimum size=3.5pt, thin, l=4mm}{}, 
  [, bdot, label=above:{$1/2$}, name=n12,
    [, wdot, label=above:{$1/3$},
      [, bdot, label={[xshift=-5pt]above:{$1/4$}}, name=n14,
        [, wdot, label=above:{$1/5$}
          [,bdot,[],[,bdot,edge={line width=1.2pt}]] [,bdot,[,bdot,edge={line width=1.2pt}][]]
        ]
        [, bdot, label=above:{$2/7$},
          edge={line width=1.2pt}
          [, edge=gray, for descendants={edge=gray},gdot,[,gdot][,gdot]] [,bdot,[],
          [,bdot,edge={line width=1.2pt}]]
        ]
      ]
      [, bdot, label=above:{$2/5$}, name=n25,
        [, bdot, label=above:{$3/8$},
          edge={line width=1.2pt},
          [,bdot,[,bdot,edge={line width=1.2pt}][]] [, edge=gray, for descendants={edge=gray},gdot,[,gdot][,gdot]]
        ]
        [, wdot, label=above:{$3/7$}
          [,bdot,[],[,bdot,edge={line width=1.2pt}]] [,bdot,[,bdot,edge={line width=1.2pt}][]]
        ]
      ]
    ]
    [, bdot, label=above:{$2/3$},
      edge={line width=1.2pt}, 
      [, gdot, edge=gray, for descendants={edge=gray}, label={[gray]above:{$3/5$}}
        [, gdot, label={[gray]above:{$4/7$}}
          [,gdot,[,gdot][,gdot]] [,gdot,[,gdot][,gdot]]
        ]
        [, gdot, label={[gray]above:{$5/8$}}
          [,gdot,[,gdot][,gdot]] [,gdot,[,gdot][,gdot]]
        ]
      ]
      [, bdot, label={[xshift=5pt]above:{$3/4$}}, name=n34,
        [, wdot, label=above:{$5/7$}
          [,bdot,[],[,bdot,edge={line width=1.2pt}]] [,bdot,[,bdot,edge={line width=1.2pt}][]]
        ]
        [, bdot, label=above:{$4/5$},
          edge={line width=1.2pt}
          [, edge=gray, for descendants={edge=gray},gdot,[,gdot][,gdot]] [,bdot,[][,bdot,edge={line width=1.2pt}]]
        ]
      ]
    ]
  ]
\end{forest}
\caption{Visualisation of the $d$-invariants of $\Sigma(2,q,r)$ on the Stern--Brocot tree. The thick edges correspond to intervals where $d>0$. Empty dots denote Brieskorn spheres with $d=0$.}
\label{fig:tree_intro}
\end{figure}

We decided to treat the $p=2$ case first as it illustrates the technique and key ideas.
For a fixed $p \geq 2$, the surgeries on $L$ correspond to elements of (the $\GL(2,\Z)$ lift of) the congruence subgroup $\Gamma(p)$. The same method can be used in this general case, but the monoids and the orbits of their actions are more arithmetically involved. It will be the subject of a subsequent paper.

As mentioned above, the importance of the $d$-invariant stems from the fact that it is a rational homology cobordism invariant, in particular, it can be used for the following question from Kirby's original list \cite[Problem 4.2]{Kirby1997_K2}:
\begin{problem}[{Kirby, 1978}]\label{problem-kirby}
    Which homology 3-spheres bound contractible $4$-manifolds or homology $4$-balls?
\end{problem}

This question attracted a lot of interest, particularly in the case of Brieskorn spheres. Various sub-families of $\Sigma(p, q, r)$ form a large portion of the examples in the current literature, starting with the celebrated example of Mazur \cite{Mazur_1960}, and followed by work of many others (see e.g. \cite{Poenaru_1960, Akbulut_Kirby_Mazur, Stern1978, CassonHarer1981, AkbulutLarson2018, Savk_2020_Brieskorn}). More generally, one can ask, for a given homology $3$-sphere, what is its class in the homology cobordism group $\Theta^3$. Bounding a homology $4$-ball is equivalent to representing the identity element of this group. For a review of the homology cobordism group, see \cite{avk2023, manolescu2018homologycobordismtriangulations}.

Our work shows that most Brieskorn spheres $\Sigma(2,q,r)$ have $d>0$ and hence they cannot bound a (rational) homology ball. As a natural next step, we plan to extend our results about the $d$-invariant to the involutive Heegaard Floer homology, and associated monotone graded roots \cite{Hendricks2017, Dai_Manolescu_2019}. These invariants can be used to show independence of families of Brieskorn spheres in the homology cobordism group. As a first step in this direction, we show that for arbitrary $p$, the matrix $\begin{inlinematrix}
    -2p &2p-1\\
    -2p-1 & 2p
\end{inlinematrix}$ preserves not only the $d$-invariant but also the Neumann--Siebenmann $\bar{\mu}$-invariant (Propositions \ref{prop:EP} and \ref{prop:mubar}).

Finally, we address the question of which Brieskorn spheres $\Sigma(2,q,r)$ have even (minimal) plumbing lattices (equivalently, the $d$ invariant attains its upper bound given by the quarter of the number of vertices in the plumbing graph). Yamada and Matsumoto classified all such $\Sigma(2,q,r)$ in \cite{Yamada_Matsumoto_1983}.  We revisit their theorem in our language of monoids in $\GL(2,\Z)$. This allows us to give explicit constraints on the value of $q/r$ (see \eqref{eq:intervals}). We expect that this method will allow us to extend their theorem to higher $p$ in future.

\subsection{Main Results}\label{subsec:main_results} For Brieskorn 3-spheres $\Sigma(p,q,r)$, $d$ is an even non-negative integer. 
The main result of the present paper is a complete classification of Brieskorn spheres $\Sigma(2,q,r)$ with $d=0$.

In our theorems below we give criteria for vanishing $d=0$ and non-vanishing $d>0$  
simultaneously (though in some cases it will not be immediately transparent from the statement that they are complementary). We provide three `pairs' of statements. 

Below we assume that $q,r$ are coprime odd integers such that $0<q<r$. 
If some positive integers  $a,b,c,d$ satisfy
\begin{equation}\label{eq:int1} 
0<\frac{a}{b}<\frac{c}{d}<1, \quad b c - a d=1, \quad b, c \text{ are even,}
\end{equation}
then the interval $[\frac{a}{b}, \frac{c}{d}]$
will be called an \emph{even interval}. A \emph{maximal even interval} is an even interval that is not strictly contained in another even interval.
Moreover, define 
\begin{equation}\label{eq:M}
\mathcal{M} \coloneqq \left\{ \frac{q}{r}\ \ \Big\vert \ \ \frac{q}{r} \ \mbox{is an element of the interior of an even interval}\right\}. \end{equation}
Associated with a maximal even interval $[\frac{a}{b}, \frac{c}{d}]$,
we define its \emph{V-fraction} $v$ (V for vanishing) as   
\begin{equation}\label{eq:int2}
v=\begin{cases} \frac{3a-c}{3b-d}  & \text{if \ $b<d$},\\[1ex]
 \frac{3c-a}{3d-b}  & \text{if \ $b>d$}. \end{cases}
\end{equation}
We define
\begin{equation*}
\mathcal{V} \coloneqq \text{ the set of all V-fractions.}
\end{equation*}
\begin{theorem}[Arithmetical characterization of $d$] \label{th:int1}\leavevmode
\begin{enumerate}[label=(\roman*), font=\upshape,]
    \item $d(\Sigma(2,q,r))>0$ if and only if $q/r \in \mathcal{M}$.
    \item $d(\Sigma(2,q,r))=0$ if and only if $q/r \in \mathcal{V}$. 
\end{enumerate}
\end{theorem}
Our main tool, surgery on the two singular fibers, is conveniently described by $2\times2$ matrices. In the $p=2$ case they belong to the congruence subgroup $\Gamma(2)$ of the modular group $\SL(2,\Z)$, or occasionally its lift in $\GL(2,\Z)$.
In the next two theorems, we use a right, respectively  left action of  $\GL(2,\Z)$. 

In the first case, 
we define an  action of  $\GL(2,\Z)$ on the set of intervals as follows.
We identify an interval $I=[\frac{e}{f},\frac{e'}{f'}]$ with the $2\times 2$ matrix 
$M_I \coloneqq \begin{inlinematrix} e & e' \\ f & f' \end{inlinematrix}$, and we define 
$I\cdot h$ $(h\in {\rm GL}(2,\Z))$ as the interval corresponding to the matrix $M_I\cdot h$. 
We also  fix the following matrices:
\[
O \coloneqq \begin{pmatrix} 0 & 1 \\ 1 & 1 \end{pmatrix}, \ \ 
g_1 \coloneqq \begin{pmatrix} 1 & 2 \\ 0 & 1  \end{pmatrix}, \ 
g_2 \coloneqq \begin{pmatrix} 1 & 2 \\ 1 & 1  \end{pmatrix}, \
g_3 \coloneqq \begin{pmatrix} 1 & 1 \\ 1 & 2  \end{pmatrix}, \ 
g_4 \coloneqq \begin{pmatrix} 1 & 0 \\ 2 & 1  \end{pmatrix}, \ 
u \coloneqq \begin{pmatrix} 0 & 1 \\ 1 & 0  \end{pmatrix}  .\] 

We denote by $\widetilde{\mathfrak{m}}$ the submonoid of $\GL(2,\Z)$ generated by $g_1,g_2,g_4$.
\begin{theorem}[Characterization via the right action of the monoid $\widetilde{\mathfrak{m}}$] \label{th:int2} \leavevmode

\begin{enumerate}[label=(\roman*), font=\upshape,]
\item The monoid $\widetilde{\mathfrak{m}}$ is free. The maximal even intervals are exactly the intervals corresponding to 
 $O \cdot \tilde{h} g_3 u^\epsilon$, 
    where $\tilde{h}\in \widetilde{\mathfrak{m}}$, 
   $\epsilon\in\{0,1\}$, and  $\epsilon=0$ exactly when $\tilde{h}$ contains an even number of $g_2$'s. 

\item[(ii-a)] The map $\widetilde{\mathfrak{m}}\ni \tilde{h}\mapsto O \tilde{h}\binom {2}{1}\in [0,1]\cap \mathbb{Q}$ is injective. (Here $ O \tilde{h}\binom {2}{1}$ means the usual matrix multiplication.)

\item[(ii-b)]  The set $\mathcal{V}$ of V-fractions is parametrized by 
 $O\tilde{h}\binom{2}{1}$, for all  $\tilde{h}\in \widetilde{\mathfrak{m}}$.
 \end{enumerate}
\end{theorem}

Finally, we consider the  left action of ${\rm GL}(2,\Z)$ on the set of column vectors $\binom{e}{f} \in \Z^2$
defined by the left multiplication. We use 
the following matrices of  ${\rm GL}(2,\Z)$.
   \begin{equation}
        h_1 = \Mat{1}{0}{2}{1},\quad h_2 = \Mat{0}{1}{1}{2}, \quad 
        h_3 = \Mat{1}{1}{1}{2}, \quad
        h_4 = \Mat{-1}{2}{-2}{3}.
    \end{equation}
We denote by $\widetilde{\mathfrak{m}}^*$ the submonoid of $\GL(2,\Z)$ generated by $h_1,h_2,h_4$.
\begin{theorem}[Characterization via the left action of the monoid $\widetilde{\mathfrak{m}}^*$] \label{th:int3} \leavevmode

\begin{enumerate}[label=(\roman*), font=\upshape,]
\item The monoid $\widetilde{\mathfrak{m}}^*$ is free.
The union of $\widetilde{\mathfrak{m}}^*$ orbits of the elements of $\left(\frac{1}{2}, \frac{2}{3}\right)\cap \mathbb{Q}$ equals $\mathcal{M}$. If $q,r$ are odd, then 
$\frac{q}{r}\in\mathcal{M}$ if and only if $d(\Sigma(2,q,r))>0$.

\item The  $\widetilde{\mathfrak{m}}^*$ orbit of $\frac{1}{3}$ is $\mathcal{V}$, and
the correspondence $\widetilde{\mathfrak{m}}^* \ni \tilde{g}\mapsto \tilde{g}\binom{1}{3}\in \mathcal{V}$ is injective. Hence the set of $q/r$ with $d(\Sigma(2,q,r))=0$ is exactly this orbit. 
\end{enumerate}
\end{theorem}
The combinatorial part of 
Theorems \ref{th:int1} and \ref{th:int2} is based on the properties of the 
Stern--Brocot tree. 
Theorems \ref{th:int2} and \ref{th:int3} are related via the matrix identities
\begin{equation}
h_i O=O g_i, \quad \text{for } i=1,2,3,4.
\end{equation}
Theorem \ref{th:int3} is related to Dehn surgeries as follows. For each $h_i$ we consider $\binom{q'}{r'}=h_i\binom{q}{r}$. Then  
$\Sigma(2,q',r')$ is a surgery manifold of $\Sigma(2,q,r)$. We express it in the language of plumbing graphs, and prove that the corresponding lattices are stably equivalent, hence 
their $d$-invariants coincide. 

For the first two matrices, the stability of the $d$-invariant follows from a theorem of Lidman and Tweedy~\cite{Lidman_2018}, which we reprove here. They corresponds to the surgery of type 
 $\Sigma\mapsto \Sigma_{-1}(K)$ on a Seifert fiber $K$. 
 The third matrix $h_4$ (which was not previously considered in~\cite{Lidman_2018}) corresponds to the $(1,3)$-surgery on the link of the second and third singular Seifert fiber.

In addition to the stability of the $d$-invariant along the $\widetilde{\mathfrak{m}}^*$ orbits, we also need a non-vanishing result:

\begin{theorem}\label{thm:nonvanishing_intro}
    \[d(\Sigma(2,q,r))>0 \quad \text{ for all } q,r \text{ with } \quad \frac12<\frac{q}{r}<\frac23.
    \]
\end{theorem}
The proof consists of two parts. Firstly, we consider the operator $S_2 \coloneqq  \begin{inlinematrix} -4& 3\\ -5& 4\end{inlinematrix} $. It preserves the interval $[\frac12,\frac23]$, reverses its orientation, and identifies  the interval $[\frac12,\frac35]$ with 
 $[\frac35,\frac23]$. Then, for 
 $\binom {q'}{r'}=S_2\binom{q}{r}$ and  $\frac12<\frac{q}{r}<\frac35$ we show that 
 the lattices of $(2,q,r)$ and $(2,q',r')$ are equivalent (in fact, this statement extends to an
 arbitrary $p\geq 2$ as well). In the second part, we prove that 
 $d(\Sigma(2,q,r))>0 $ for any
 $\frac12<\frac{q}{r} \leq \frac35$.
The key step is a certain lattice embedding and a result of Lecuona and Lisca \cite[Lemma 3.3]{LL}.

As a topological application, we obtain:
\begin{corollary}
    If $q/r \in \mathcal{M}$, the Brieskorn sphere $\Sigma(2,q,r)$ does not bound a rational homology ball. In particular, this holds in the interval $(\frac12,\frac23)$. 
\end{corollary}

\subsection{Paper Outline} In \cref{subsec:main_results}, we gave a summary of the main results, defining most of the transformations we will use in later sections to verify our results about the $d$-invariant. The reader interested primarily in using our calculations to find the $d$-invariants for specific values of $(p, q, r)$ may find in \cref{sec:examples} worked calculations for several specific values of $(2, q, r)$. In \cref{sec:even_plumbings} one can find a more complete description of the motivations of the present work in terms of the previous results of Yamada and Matsumoto on \emph{even plumbings},~\cite{Yamada_Matsumoto_1983}, which likewise provides a large number of concrete examples.

In \cref{sec:sb-combinatorics}, we discuss the \emph{Stern--Brocot tree}, the primary combinatorial tool used to find the tuples $(q, r)$ for which $d( \Sigma(2, q, r) ) = 0$. \cref{sec:brieskorn_notation} establishes the notation for lattices and singularities which will be used in verifications which follow. \cref{s:Gs} contains the key lattice computations which will be used in \cref{sec:proof_main_thms} to show that the operators $h_1, h_2,$ and $h_4$ preserve the $d$-invariant. In \cref{sec:equivalent_lattices}, we give the proof that the transformation $S_p$, which reflects a key interval in the Stern--Brocot sequence, still preserves the lattice corresponding to $(p, q, r)$; we also discuss in somewhat more detail the surgery description of the transformations $h_1, h_2,$ and $h_4$. Finally, in \cref{s:6} and \cref{sec:proof_main_thms}, we use the preceding lattice calculations, along with a result of Lecuona and Lisca, to verify our main claim.

\subsection*{Acknowledgments} We are grateful to Brandon J. Robinson for aiding computer computations at the start of the project. We thank Nicola Tarasca, Allison Moore and Lisa Piccirillo for useful discussions. S. Harichurn was supported by the FirstRand FNB 2020 Fund Education Scholarship. A. N\'emethi was partially supported by NKFIH Grant ``\'Elvonal (Frontier)'' KKP 144148. I. Khan and
J. Svoboda were supported by the Simons Grant \emph{New structures in Low-Dimensional Topology}. 

\subsection*{Note on AI use} All mathematical arguments and ideas appearing in the paper are due to the authors. We used LLMs for implementing the algorithm for computing the $d$-invariant from \cite{Nemethi2005}, and for helping with the TikZ graphics. Before arXiv submission, we checked the manuscript with an LLM and corrected a few minor inaccuracies. The code for the $d$-invariants is available upon request. 

\section{Some combinatorics of the Stern--Brocot sequences}\label{sec:sb-combinatorics}

Recall that the Stern--Brocot sequences $\{\SB_n\}_{n\geq 0}$ of the interval $[0,1]$
are defined inductively as follows. Each $\SB_n$ is an increasing sequence of rational numbers in the interval $[0,1]$. 
We start with $\SB_0=\{\frac{0}{1}, \frac{1}{1}\}$.
Then the sequence $\SB_{n+1}$ consists of the elements of $\SB_n$ together with some new ones: 
for each adjacent pair $\frac{a}{b}<\frac{c}{d}$ in $\SB_n$, we add the fraction $\frac{a+c}{b+d}$ between them. 
E.g., $\SB_1=\{\frac{0}{1}, \frac{1}{2},\frac{1}{1}\}$ and $\SB_2=\{\frac{0}{1}, \frac{1}{3},\frac{1}{2},
 \frac{2}{3}, \frac{1}{1}\}$. 

If $\frac{a}{b}<\frac{c}{d}$ is an adjacent pair in $\SB_n$, then $ad-bc=-1$. We call $[\frac{a}{b},\frac{c}{d}]$
an \emph{elementary interval} of $\SB_n$.  
We sometimes denote a sequence of consecutive fractions $\frac{a}{b}<\frac{c}{d}<\frac{e}{f}<\cdots$ in $\SB_n$ by the 
matrix $\begin{inlinematrix} a & c & e &\cdots \\b & d & f &\cdots \end{inlinematrix}$.
For such a matrix, the $({\rm mod}\ 2)$ matrix means that we replace each entry by its remainder $({\rm mod}\ 2)$.

One may also construct the classical Stern--Brocot tree $\mathcal{T}$ from the Stern--Brocot sequences as follows: 
  if $\frac{a}{b}<\frac{c}{d}<\frac{e}{f}$ are adjacent fractions in $\SB_n$ such that 
  $\frac{c}{d}$ is newly created in $\SB_{n}$, then in this tree $\frac{c}{d}$
  is connected by an edge with the fractions $\frac{a+c}{b+d}$ and $\frac{c+e}{d+f}$ of $\SB_{n+1}$.
  Hence, each fraction (except $\frac{0}{1}$ and $\frac{1}{1}$) appears as a vertex
  of the tree, it has one ancestor and two `children'. See \cref{fig:tree}.

  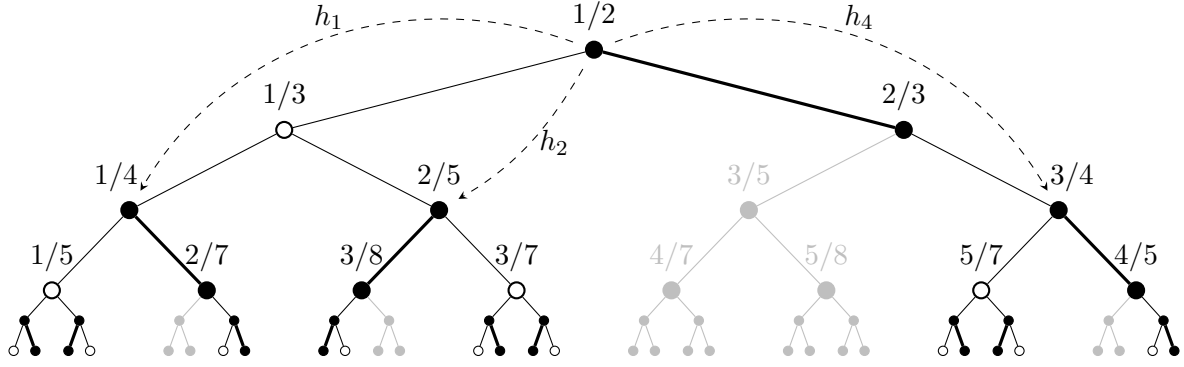
\begin{figure}
\centering
\colorlet{gray}{gray!50}
\begin{forest}
  for tree={
    circle,
    draw,
    thick,
    minimum size=6pt,   
    inner sep=0pt,      
    l sep=4mm,          
  },
  where level<=2{s sep=9mm}{},
  where level=3{s sep=3mm, l sep=1mm}{},     
  where level=4{s sep=1.5mm, l sep=1mm}{},   
  where level>=4{minimum size=3.5pt, thin, l=4mm}{}, 
  [, bdot, label=above:{$1/2$}, name=n12,
    [, wdot, label=above:{$1/3$},
      [, bdot, label={[xshift=-5pt]above:{$1/4$}}, name=n14,
        [, wdot, label=above:{$1/5$}
          [,bdot,[],[,bdot,edge={line width=1.2pt}]] [,bdot,[,bdot,edge={line width=1.2pt}][]]
        ]
        [, bdot, label=above:{$2/7$},
          edge={line width=1.2pt}
          [, edge=gray, for descendants={edge=gray},gdot,[,gdot][,gdot]] [,bdot,[],
          [,bdot,edge={line width=1.2pt}]]
        ]
      ]
      [, bdot, label=above:{$2/5$}, name=n25,
        [, bdot, label=above:{$3/8$},
          edge={line width=1.2pt},
          [,bdot,[,bdot,edge={line width=1.2pt}][]] [, edge=gray, for descendants={edge=gray},gdot,[,gdot][,gdot]]
        ]
        [, wdot, label=above:{$3/7$}
          [,bdot,[],[,bdot,edge={line width=1.2pt}]] [,bdot,[,bdot,edge={line width=1.2pt}][]]
        ]
      ]
    ]
    [, bdot, label=above:{$2/3$},
      edge={line width=1.2pt}, 
      [, gdot, edge=gray, for descendants={edge=gray}, label={[gray]above:{$3/5$}}
        [, gdot, label={[gray]above:{$4/7$}}
          [,gdot,[,gdot][,gdot]] [,gdot,[,gdot][,gdot]]
        ]
        [, gdot, label={[gray]above:{$5/8$}}
          [,gdot,[,gdot][,gdot]] [,gdot,[,gdot][,gdot]]
        ]
      ]
      [, bdot, label={[xshift=5pt]above:{$3/4$}}, name=n34,
        [, wdot, label=above:{$5/7$}
          [,bdot,[],[,bdot,edge={line width=1.2pt}]] [,bdot,[,bdot,edge={line width=1.2pt}][]]
        ]
        [, bdot, label=above:{$4/5$},
          edge={line width=1.2pt}
          [, edge=gray, for descendants={edge=gray},gdot,[,gdot][,gdot]] [,bdot,[][,bdot,edge={line width=1.2pt}]]
        ]
      ]
    ]
  ]
  \draw[->, dashed, >=stealth, shorten <= 5pt, shorten >= 5pt] (n12) to[out=160, in=60, looseness=1]
    node[midway, above, align=center] {$h_1$} (n14);
  \draw[->, dashed, >=stealth,shorten <= 5pt, shorten >= 5pt] (n12) to[out=240, in=30, looseness=1]
    node[right, align=center] {$h_2$} (n25);
  \draw[->, dashed, >=stealth,shorten <= 5pt, shorten >= 5pt] (n12) to[out=20, in=120, looseness=1, shorten <= 50pt, shorten >= 15pt]
    node[midway, above, align=center] {$h_4$} (n34);
\end{forest}
\caption{Visualisation of the $d$-invariants of $\Sigma(2,q,r)$ on the Stern--Brocot tree. Empty dots denote V-fractions, for which $d=0$. Thick edges are the newly created black intervals. Shaded are the inherited black intervals. We have $d>0$ in the interior of any black interval. Matrices $h_1,h_2,h_4$ map the whole tree onto 3 subtrees, preserving the $d$-invariant.}
\label{fig:tree}
\end{figure}

\subsection{`Black intervals' and `V-fractions' in $[0,1]$.}

Below, for any fraction $\frac{a}{b}$, we assume $\gcd(a,b)=1$.
\begin{lemma}\label{lem:SB1}In $\SB_{2\ell}, \ell \geq 0$, any elementary interval $
\begin{inlinematrix} a & c \\b & d  \end{inlinematrix}$ $(\rm{mod} \ 2) $
has one of the forms
\[O_{\I} \coloneqq \begin{pmatrix} 0 & 1 \\ 1 & 1  \end{pmatrix}, \ \  
E \coloneqq \begin{pmatrix} 1 & 0 \\ 0 & 1  \end{pmatrix}, \ \  
O_{\II} \coloneqq \begin{pmatrix} 1 & 1 \\ 1 & 0  \end{pmatrix}\ \ \ \ ({\rm mod}\ 2).
\]
\end{lemma}
It is convenient to analyse every two steps $\SB_{2\ell}\rightsquigarrow \SB_{2\ell+2}$ at the same time.
If an elementary interval of type $E$ modulo 2 appears in some 
$\SB_{2\ell}$, it is called a \emph{black interval}. The intervals of type $O_{\I}$ and $O_{\II}$ modulo 2 are called \emph{white intervals}.
In the  $\SB_{2\ell}\rightsquigarrow \SB_{2\ell+2}$ procedure, 
the white elementary intervals $O_{\I}$ and $O_{\II}$ subdivide into four intervals as follows:
\[(\text{Type } \I): \ \ 
O_{\I} \coloneqq \begin{pmatrix} 0 & 1 \\ 1 & 1  \end{pmatrix}   \mapsto  D_{\I} \coloneqq 
\begin{pmatrix} 0 &1 & 1 & 0 & 1 \\ 1 & 1 & 0 & 1 & 1\end{pmatrix} \ \ \ ({\rm mod} \ 2);\]
\[
(\text{Type } \II): \ \ 
O_{\II} \coloneqq \begin{pmatrix} 1 & 1 \\ 1 & 0  \end{pmatrix}   \mapsto D_{\II} \coloneqq 
\begin{pmatrix} 1 &1 & 0 & 1 & 1 \\ 1 & 0 & 1 & 1 & 0\end{pmatrix} \ \ \ ({\rm mod} \ 2).\]
Along the step $\SB_{2\ell}\rightsquigarrow \SB_{> 2\ell}$, any elementary interval obtained by the 
subdivision of a black elementary interval will be called \emph{inherited black}
(even if their form modulo 2 is not $E$).
On the other hand, along the step $\SB_{2\ell}\rightsquigarrow \SB_{2\ell+2}$, 
any white elementary interval will be divided into four intervals, 
and the unique one which $E$ modulo 2, will 
be a black elementary interval of $\SB_{2\ell+2}$, which we call \emph{newly created} ones. It is the third (resp. second) interval 
in the case of Type I (resp. Type II) -- see the matrices $D_{\I}$ and $D_{\II}$. 

\begin{figure}[ht]
    \centering
    \input{graphics/type-i-ii/interval_subdivision}
    \caption{Subdivision modulo 2 of the white elementary intervals $O_{\I}$ and $O_{\II}$. V-fractions are indicated by \hollowvertex{}; bold edges indicate black intervals.}
    \label{fig:interval_subdivision}
\end{figure}
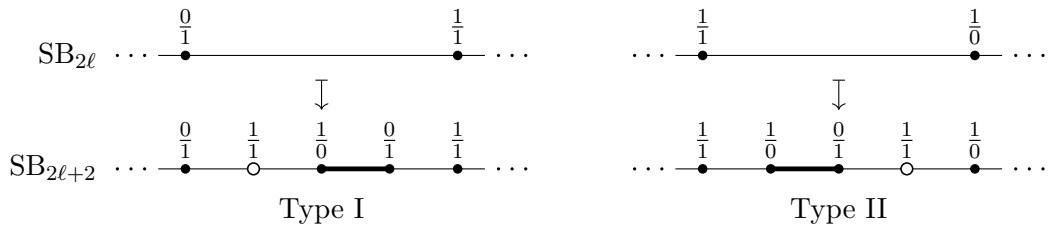

Also, along such a step, we also choose a single fraction  from the three newly created rational numbers: 
the first one for the Type I and the third one for the Type II. We call them \emph{V-fractions}. See \cref{fig:interval_subdivision}.

\begin{example}
    $\SB_0$ has no black interval and no V-fraction.
$\SB_2$ has one black interval, namely $[\frac{1}{2}, \frac{2}{3}]$,  and one 
V-fraction,  $\frac{1}{3}$. In $\SB_4$, three newly created intervals appear:
 $[\frac{1}{4}, \frac{2}{7}]$,  $[\frac{3}{8}, \frac{2}{5}]$,  $[\frac{3}{4}, \frac{4}{5}]$.
 The corresponding V-fractions are  $\frac{1}{5}$, $\frac{3}{7}$ and  $\frac{5}{7}$.

\begin{center}
\begin{tikzpicture}[x=11.5cm, y=1.2cm,
      dot/.style       = {circle, fill, inner sep=0pt, minimum size=3.6pt},
      hollow/.style    = {circle, draw, fill=white, inner sep=0pt,
                          minimum size=4.4pt, line width=0.6pt},
      lab/.style       = {above=3pt, inner sep=1pt},
      subint/.style    = {line width=1.8pt},
      inherited/.style = {subint}]   

  \foreach \y/\lev in {2/0, 1/2, 0/4} {
    \draw (0,\y) -- (1,\y);
    \node[left=6pt] at (0,\y) {$\SB_{\lev}$};
  }

  \draw[subint] (1/2,1) -- (2/3,1);
  \draw[subint] (1/4,0) -- (2/7,0)  (3/8,0) -- (2/5,0)  (3/4,0) -- (4/5,0);
  \draw[inherited] (1/2,0) -- (2/3,0);

  \foreach \y/\p/\q in {
      2/0/1, 2/1/1,
      1/0/1, 1/1/3, 1/1/2, 1/2/3, 1/1/1,
      0/0/1, 0/1/5, 0/1/4, 0/2/7, 0/1/3, 0/3/8, 0/2/5, 0/3/7, 0/1/2,
      0/4/7, 0/3/5, 0/5/8, 0/2/3, 0/5/7, 0/3/4, 0/4/5, 0/1/1} {
    \node[dot] at (\p/\q,\y) {};
    \node[lab] at (\p/\q,\y) {$\frac{\p}{\q}$};
  }

  \foreach \y/\p/\q in {1/1/3, 0/1/5, 0/1/3, 0/3/7, 0/5/7}
    \node[hollow] at (\p/\q,\y) {};
\end{tikzpicture}
\end{center}
 
\end{example}

\begin{remark}
    Note that black intervals and V-fractions appear in pairs (having the same white interval ancestor). In each $\SB_{2\ell}$ $(\ell\geq 1)$, there are $3^{\ell-1}$ newly created black intervals and the same number of V-fractions.
\end{remark}

The corresponding picture for the Stern--Brocot tree $\mathcal{T}$ is shown in \cref{fig:tree_subdivision}. 
Note that white intervals may not always appear as edges on $\mathcal{T}$, but at least one of their endpoints will appear as a vertex. The only exception to this is the white interval $[\frac{0}{1}, \frac{1}{1}]$ which, after a Type I move, produces (per the \cref{fig:tree_subdivision} convention):
 \begin{equation}
     \TypeIinline.
 \end{equation}

Recall from the introduction that if we have two rational numbers $0<\frac{a}{b} < \frac{c}{d} <1$ with $ad-bc=-1$ and $b,c$ even numbers, we call the corresponding interval $[\frac{a}{b},\frac{c}{d}]$ an \emph{even interval}. A \emph{maximal even interval} is an even interval maximal w.r.t. the inclusion. 
 \begin{lemma}[First characterization of black intervals]\label{lem:SB2} Any maximal even interval is a black interval at some step $\SB_{2\ell}$ of the 
 Stern--Brocot sequences. 
 \end{lemma}
 \begin{proof}
    It is known that any two rational numbers $0<\frac{a}{b}< \frac{c}{d}<1$ with $ad-bc=-1$ appear as adjacent fractions in $\SB_n$ for some $n$. Induction and the construction of the sequences of $\SB_n$ show that
    if $b$ and $c$ are even, then so is $n$. 
 \end{proof}

\begin{lemma}\label{lem:SB3}
\leavevmode
\begin{enumerate}[label=(\alph*)] 
\item The V-fractions as constructed above are all distinct.

\item Let $\mathcal{V}$ denote the set of V-fractions. Then
$(0,1)\cap\mathbb{Q}= \mathcal{V}\sqcup ( \cup\{\text{black intervals}\})$.
\end{enumerate}
 \end{lemma}
 
 \begin{proof}
 \leavevmode
 \begin{enumerate}[label=(\alph*), listparindent=0pt]
     \item Any rational number appears in the $\{\SB_n\}$-sequences only once.

     \item First we show that $\mathcal{V}\cap ( \cup\{\text{black intervals}\})=\emptyset$.
     Indeed, assume the contrary and set $v\in\mathcal{V}\cap I$, where $I$ is a black interval. Assume that the V-fraction $v$ was created in $\SB_n$, while $I$ in $\SB_m$.  
     Since V-fractions appear only in the interior of a white interval, 
     $I$ cannot be created in an earlier step than the V-fraction $v$, i.e. $m>n$. 
     But as any black interval is strictly in the interior of a white interval, it cannot contain the V-fractions created earlier, i.e. $n>m$.

 On the other hand, 
$[0,1]\cap\mathbb{Q}\subset \mathcal{V} \cup ( \cup\{\mbox{black intervals}\})$ since 
 along the step $\SB_{2\ell}\rightsquigarrow \SB_{ 2\ell+2}$ we have the equality
\begin{align*}
   \SB_{2\ell+2}&=\SB_{2\ell}\sqcup \{\mbox{newly created V-fractions}\}\\&\sqcup
     \{\mbox{endpoints of the newly created black intervals}\}, 
\end{align*}
     and every element of $(0,1)\cap\mathbb{Q}$ appears in some $\SB_{2\ell}$. \qedhere
\end{enumerate}
 \end{proof}
\begin{lemma}[First characterization of V-fractions]\label{lem:SB4}
If a newly created interval  $[\frac{a}{b},  \frac{c}{d}]$
and a V-fraction $v$ form a pair (that is, if they have the same white interval ancestor)
then 
\begin{equation}\label{eq:3ac}
v=\begin{cases} \frac{3a-c}{3b-d}  & \text{if \ $b<d$},\\[1ex]
 \frac{3c-a}{3d-b}  & \text{if \ $b>d$}. \end{cases}
\end{equation}
\end{lemma} 
\begin{proof}
    From the black interval, one can recover the ancestor white interval, hence the paired V-fraction as well. Type I (resp. Type II)
    corresponds to $b<d$ (resp. $b>d$). 
\end{proof}

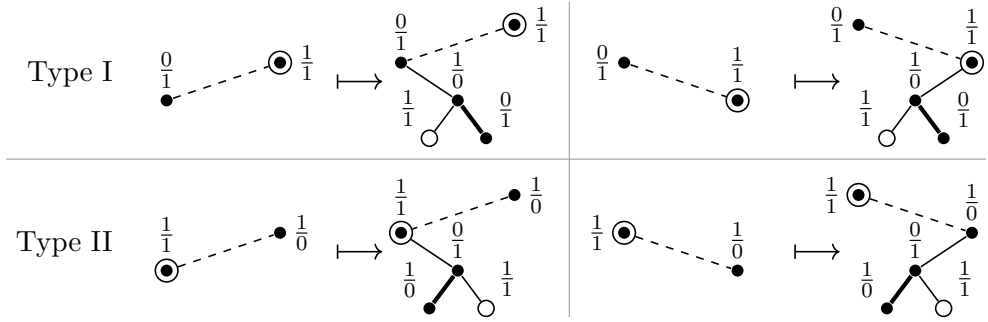
\begin{figure}[htbp]
    \centering
   \input{graphics/type-i-ii/sb_tree_type_i_ii}
\caption{ A $(\operatorname{mod} \  2)$ schematic of how Type I and II moves applied to white intervals of $\SB_{2\ell}$ ($\ell \geq 1$) appear on the Stern--Brocot tree $\mathcal{T}$. If a white interval is not represented by an edge of $\mathcal{T}$ (or additionally has an endpoint in $\{\frac{0}{1}, \frac{1}{1}\}$), the corresponding edge on the right hand side (additionally the vertex corresponding to the endpoint) does not appear on the tree. The symbol, \circledvertex{}, denotes the possibility of either a standard or V-fraction.   
}
\label{fig:tree_subdivision}
\end{figure}

\subsection{Second characterization of $\mathcal{V}$ and of the set of 
black intervals}\label{ss:2.2}

When we pass from $\SB_{n}$ to $\SB_{n+1}$, each elementary interval of $\SB_{n}$ is divided into two elementary intervals of $\SB_{n+1}$. This can be described by two operators
\[
L \coloneqq \begin{pmatrix} 1 & 1 \\ 0 & 1  \end{pmatrix} \ \ \mbox{and} \ \ 
R \coloneqq \begin{pmatrix} 1 & 0 \\ 1 & 1  \end{pmatrix}.\] 
If the matrix $I=\begin{inlinematrix} a & c \\ b & d  \end{inlinematrix} $ represents 
the interval  $[\frac{a}{b},  \frac{c}{d}]$ in $\SB_n$, it is divided into
$\begin{inlinematrix} a & a+c & c \\ b & b+d & d  \end{inlinematrix}$
in $\SB_{n+1}$. Then the two intervals of the subdivision correspond to the matrices given by the right multiplication by $L$ and $R$, respectively: 
\[
\begin{pmatrix} a & a+c  \\ b & b+d  \end{pmatrix} = \begin{pmatrix} a & c \\ b & d  \end{pmatrix}
\begin{pmatrix} 1 & 1 \\ 0 & 1  \end{pmatrix}
\ \ \mbox{resp.} \ \ 
\begin{pmatrix}  a+c &c   \\  b+d & d \end{pmatrix} =\begin{pmatrix} a & c \\ b & d  \end{pmatrix}
\begin{pmatrix} 1 & 0 \\ 1 & 1  \end{pmatrix}.
\]
When we pass from 
$\SB_{2\ell}$ to $\SB_{2\ell+2}$, each elementary interval of $\SB_{2\ell}$ is divided into four elementary intervals of $\SB_{2\ell+2}$. The following four operators in ${\rm SL}(2,\Z)$ represent the four newly created intervals, in order from left to right:
\[
G_1 \coloneqq LL=\begin{pmatrix} 1 & 2 \\ 0 & 1  \end{pmatrix}, \
G_2 \coloneqq LR=\begin{pmatrix} 2 & 1 \\ 1 & 1  \end{pmatrix}, \
G_3 \coloneqq RL=\begin{pmatrix} 1 & 1 \\ 1 & 2  \end{pmatrix}, \
G_4 \coloneqq RR=\begin{pmatrix} 1 & 0 \\ 2 & 1  \end{pmatrix}.
\]

In what follows we also need the following two matrices:
\[ O \coloneqq \begin{pmatrix} 0 & 1 \\ 1 & 1  \end{pmatrix} \
\ \mbox{and} \ \ 
U \coloneqq  \begin{pmatrix} 0 & 1 \\ 1 & 0  \end{pmatrix}. 
\]
Since all elementary intervals of the different levels $\{\SB_{2\ell}\}_{\ell\geq 0}$
form a division of $[\frac{0}{1},  \frac{1}{1}]$
 (represented by $O$),
all these elementary intervals are parametrized as
$O\cdot g$, where $g$ 
is an element of the monoid $\mathfrak{M}$ generated by $G_1,G_2,G_3,G_4$ in ${\rm SL}(2,\Z)$. This action also shows that  $\mathfrak{M}$ is a free monoid. 
Let $\mathfrak{M}_\ell$ be the set of elements of $\mathfrak{M}$ containing $\ell$ letters. 
Then, in fact, the lexicographic order (from left to right) of monomials (elements) of $\mathfrak{M}_\ell$ corresponds to the order of the elementary intervals determined by  $\SB_{2\ell}$ (ordered from left to right). (E.g. in $\SB_4$, $G_1G_1<G_1G_2<G_1G_3<G_1G_4 <G_2G_1<\cdots$, corresponds to 
$[\frac{0}{1},  \frac{1}{5}]<  [\frac{1}{5},  \frac{1}{4}]<
[\frac{1}{4},  \frac{2}{7}]< \dots$).

Next, \emph{we wish to identify the black elementary intervals}. We have to analyse two cases:
if $I$ is a white interval of $\SB_{2\ell}$, then in the Type I case,
the newly created black interval (in $I$) is $I\cdot G_3$, while in the Type II case
it is $I \cdot G_2$. We would like to unify these two cases. Note that Type II
appears in the case when the interval $I$ is  realized as a $G_2$-image. We say that the operator $G_2$ reverses the ordering of the consecutive intervals. Therefore, 
we wish to replace the operator $G_2$ by $G_2U$. 

Note also that 
$\begin{inlinematrix} a & c \\ b & d  \end{inlinematrix}\cdot U=
\begin{inlinematrix} c & a \\ d & b  \end{inlinematrix}$ 
represents the very same interval with reversed orientation. 
We have the immediate identities of operators: $UL=RU$, $UR=LU$, hence 
$U G_1 = G_4 U$ and $U G_2 = G_3 U$.

We first consider the monoid generated in ${\rm GL}(2,\Z)$ by the elements 
\[g_1 \coloneqq G_1, \quad 
g_2 \coloneqq G_2 U = \begin{pmatrix} 1 & 2 \\ 1 & 1  \end{pmatrix}, \quad
g_3 \coloneqq G_3, \quad 
g_4 \coloneqq G_4, \quad 
u \coloneqq U. 
\]
Let 
$\mathfrak{m}_u$ be its submonoid consisting of the monomials where the
$ \{ \mbox{number of}\  g_2\mbox{'s}+  \mbox{number of} \ u\mbox{'s}\}$ is even.
Note that in $\mathfrak{m}_u$, we have the identities 
\begin{equation}\label{eq:com} 
u g_1 = g_4 u,\quad u g_4 = g_1 u, \quad u g_2 = g_3, \quad u g_3 = g_2 , \quad u^2 = 1.
\end{equation}
Then any element of $\mathfrak{M}$, by the substitution 
$G_1 \mapsto g_1 $, $G_2 \mapsto g_2 u $, 
$G_3 \mapsto g_3 $, $G_4 \mapsto g_4 $,
transforms into an element of $\mathfrak{m}_u$. 
Even more, if we denote by $\mathfrak{m}$ the monoid in ${\rm GL}(2,\Z)$ generated by 
$g_1, g_2, g_3, g_4$, then, by (\ref{eq:com}), any element of $\mathfrak{m}_u$ can be written in a unique form as 
$hu^\epsilon $, where $h\in \mathfrak{m}$ and $\epsilon =1$ when $h$ contains an odd number of $g_2$'s, and $\epsilon =0$ otherwise. Since any $h\in \mathfrak{m} $ determines
$hu^\epsilon\in \mathfrak{m}_u$ uniquely, the monoids $\mathfrak{m}$ and $\mathfrak{m}_u$ are isomorphic. Moreover, they
parametrize all the elementary intervals via $O\cdot hu^\epsilon $ (see \cref{fig:sb_characterization_two}). 
The advantage of this second reparametrization is the following:
\begin{lemma}[Second characterization of black intervals]\label{lem:SB6}

We write $\widetilde{\mathfrak{m}}$ for the free monoid generated by $g_1,g_2,g_4$. Then 
$O \cdot hu^\epsilon $ is a subinterval of an elementary black interval if and only if $h\in\mathfrak{m}$ contains the letter $g_3$. 

More precisely,  all the (non-inherited) elementary black intervals have the form 
 $O \cdot \tilde{h} g_3 u^\epsilon$, 
    where $\tilde{h}\in \widetilde{\mathfrak{m}}$, 
   $\epsilon\in\{0,1\}$, and $\epsilon=0$ exactly when the number of $g_2$'s in $\tilde{h}$ is even. 
    In the $\epsilon=0$ case the black interval is the output of a Type I step, while in the $\epsilon=1$ case of a Type II step. 
\end{lemma}
Next, we parametrise the V-fractions in a similar way. 
If $O\cdot \tilde{h} g_3$ is a black interval created by a Type I step,
 then $O\cdot \tilde{h} g_2 u$ is its previous interval whose left boundary 
 (first column) is the V-fraction $v$, paired to $O\cdot \tilde{h} g_3$. 
 That is, $v=O\cdot \tilde{h} g_2 u\binom{1}{0}=
 O\cdot \tilde{h}\binom{2}{1}$. As an example, in $\SB_4$ on \cref{fig:sb_characterization_two}, one sees the black interval $O\cdot g_4g_3$ paired to the V-fraction $\frac{5}{7}$ which is the left boundary of $O\cdot g_4g_2 u$.

 If $O\cdot \tilde{h} g_3 u$ is a black interval created by a Type II step,
 then $O\cdot \tilde{h} g_2 $ is its next interval  whose right boundary 
 (second  column) is the V-fraction $v$, paired to $O\cdot \tilde{h}g_3u$. 
 That is, $v=O\cdot \tilde{h} g_2 \binom{0}{1}=
 O\cdot \tilde{h}\binom{2}{1}$.

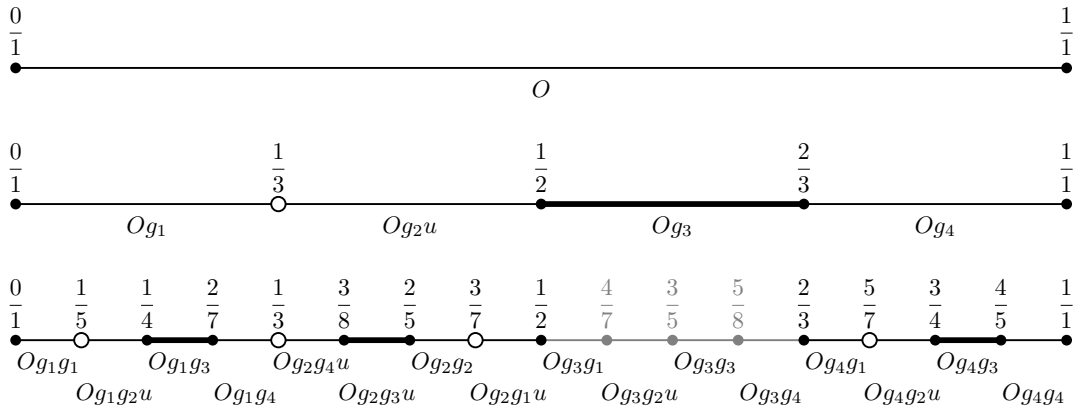
\begin{figure}[b]
    \centering
    \input{graphics/type-i-ii/sb_characterization_two}
    \caption{Parameterization of intervals of $\SB_0, \SB_2, \SB_4$ by elements of the monoid $\mathfrak{m}_u$.}
    \label{fig:sb_characterization_two}
\end{figure}
 
\begin{corollary}[Second characterization of V-fractions]\label{cor:V} \leavevmode

\begin{enumerate}[label=(\alph*)] 
\item The map $\widetilde{\mathfrak{m}}\ni \tilde{h}\mapsto O\cdot \tilde{h}\binom {2}{1}\in [0,1]\cap \mathbb{Q}$ is injective.

\item The set $\mathcal{V}$ of V-vertices is parametrized by 
$O\cdot \tilde{h}\binom {2}{1}$, for all $\tilde{h}\in \widetilde{\mathfrak{m}}$. 
 All the V-vertices in $\SB_{2\ell+2}$ are parametrized as 
$O\cdot \tilde{h}\binom {2}{1}$, where 
 $\tilde{h}\in \widetilde{\mathfrak{m}}$ has $\ell$ letters. 
\end{enumerate}
\end{corollary}
\begin{remark}
 Since $g_3 \binom{3}{-1}=g_3 u \binom{-1}{3}=\binom{2}{1}$, 
 Lemma \ref{lem:SB4} is compatible with Corollary \ref{cor:V}(b). 
\end{remark}

\begin{example}
In $\SB_2$ we have $O\cdot 1\binom{2}{1}=\binom{1}{3}$;
in $\SB_4$ we have $O\cdot g_1 \binom{2}{1}=\binom{1}{5}$,
 $O\cdot g_2 \binom{2}{1}=\binom{3}{7}$, $O\cdot g_4 \binom{2}{1}=\binom{5}{7}$.
\end{example}

\subsection{Third characterization of $\mathcal{V}$ and of the set of 
black intervals}\label{ss:2.3}

In subsection \ref{ss:2.2} we described the V-fractions and the black intervals in terms of the \emph{right} multiplication 
of $O$ by elements of ${\rm GL}(2,\Z)$ in the form 
$O\cdot \tilde{h}\binom{2}{1}$ and 
$O\cdot \tilde{h} g_3 u^\epsilon$.
 In this subsection we will transport these results to a \emph{left}  action of a submonoid of ${\rm GL}(2,\Z)$.
 The new generators (matrices) will be compatible with 
 certain surgery coefficients of the 3-manifolds $\Sigma(2,q,r)$ and
 with the corresponding topological language. 
 We define 
\[h_1 \coloneqq \begin{pmatrix} 1 & 0 \\ 2 & 1  \end{pmatrix}, \
h_2 \coloneqq \begin{pmatrix} 0 & 1 \\ 1 & 2  \end{pmatrix}, \
h_3 \coloneqq \begin{pmatrix} 1 & 1 \\ 1 & 2  \end{pmatrix}, \
h_4 \coloneqq \begin{pmatrix} -1 & 2 \\ -2 & 3  \end{pmatrix}.
\]
The left multiplication by $h_1, h_2, h_3, h_4$ 
sends the interval $[0,1]$ to the intervals $[0, \frac{1}{3}]$,
$[\frac{1}{3}, \frac{1}{2}]$, $[\frac{1}{2}, \frac{2}{3}]$,
$[\frac{2}{3}, \frac{1}{1}]$, in such a way that 
$h_1, h_3, h_4$ preserve the orientation of the intervals, while $h_2$ reverses it. Indeed, the left multiplication $h_1\cdot O$ of $O$ by $h_1$ provides the matrix $\mat{0}{1}{1}{3}$ representing the interval $[0, \frac{1}{3}]$, and we have similar identities for the others. 

In particular, if $\mathfrak{m}^*$ denotes the free monoid in 
 ${\rm GL}(2,\Z)$ generated by $h_1, h_2, h_3, h_4$, and 
 $\mathfrak{m}^*_\ell$ those elements which contain $\ell$ letters, 
 then 
 the elements $g\cdot O$ with $g\in\mathfrak{m}^*_\ell$ parametrize the elementary intervals 
of $\SB_{2\ell}$ (some of them with reversed orientation).
E.g., the monomials 
\begin{equation*}
h_1h_1,  h_1h_2,  h_1h_3,  h_1h_4, 
h_2h_4,  h_2h_3,  h_2h_2,  h_2h_1, 
h_3h_1,  h_3h_2,  h_3h_3,  h_3h_4,
h_4h_1,  h_4h_2,  h_4h_3,   h_4h_4
\end{equation*}
parametrize the elementary intervals of $\SB_4$ in the order from left to right (cf. \cref{fig:sb_characterization_two}). 

In fact, we have the following commuting relations, which 
connect the left and right multiplications, and transfer all the results of the previous subsection in the new context of the
$h_i$-multiplications:
\begin{equation}\label{eq:com2}h_i O=O g_i, \quad \text{for  } i=1,2,3,4.\end{equation}
Let $\widetilde{\mathfrak{m}}^*$ be the free submonoid of 
 ${\rm GL}(2,\Z)$ generated by $h_1, h_2, h_4$. Consider the left multiplication of $\widetilde{\mathfrak{m}}^*$ of column vectors (and identify a rational number $\frac{a}{b}$ with the column vector $\binom{a}{b}$). 
 \begin{proposition}[Third characterization of the black intervals and V-fractions]\label{prop:Maction}
\leavevmode
\begin{enumerate}[label=(\alph*)] 
\item The union of $\widetilde{\mathfrak{m}}^*$ orbits of elements $[\frac{1}{2}, \frac{2}{3}]\cap \mathbb{Q}$ equals $\cup\{\mbox{black intervals}\}\cap 
\mathbb{Q}$. 

\item The $\widetilde{\mathfrak{m}}^*$ orbit of $\frac{1}{3}$
 is the set of V-fractions, i.e. $\widetilde{\mathfrak{m}}^*\binom{1}{3}=\mathcal{V}$. 
 The correspondence $\widetilde{\mathfrak{m}}^* \ni \tilde{g}\mapsto \tilde{g}\binom{1}{3}\in \mathcal{V}$ is injective.

\item For any $h_i$ ($i=1,2,4$) and $\frac{a}{b}\in\SB_{2\ell}\cap\mathcal{V}$ 
 one has $h_i \binom{a}{b}\in \SB_{2+2\ell}\cap \mathcal{V}$.
\end{enumerate}
\end{proposition}

\begin{proof} The expression $O\cdot \tilde{h}$ for any 
$\tilde{h}\in \widetilde{\mathfrak{m}}$ transforms by (\ref{eq:com2})
 into an expression $\tilde{g}\cdot O$ for some $\tilde{g}\in 
 \widetilde{\mathfrak{m}}^*$. Then use the identities
 \[O\cdot \tilde{h} g_3 u^\epsilon =\tilde{g}\cdot O\cdot g_3 u^\epsilon=
 \tilde{g}h_3\cdot Ou^\epsilon =\tilde{g}\cdot
 \begin{pmatrix} 1 & 2 \\ 2 & 3  \end{pmatrix}u^\epsilon ; \ \ 
 O\cdot \tilde{h}\binom{2}{1}= \tilde{g}\cdot O \binom {2}{1}=\tilde{g}\binom{1}{3}.
 \]   
\end{proof}

\begin{remark}
  Recall the classical Stern--Brocot tree $\mathcal{T}$ defined at the start of this section. For any vertex $u$ of the tree, let $\mathcal{T}_u$ denote the (full) subtree consisting of $u$ and all its descendants. 

(a)   Assume that $v$ is a newly created V-vector in some $\SB_{2\ell}$ with $\ell>0$. Recall that it is 
  the second vector of $D_{\I}$ in case Type I, and the fourth vector of $D_{\II}$ in case Type II. In both cases let $u_v$ denote the third vector of $D_{\I}$ respectively $D_{\II}$,
  it is a newly created fraction in $\SB_{2\ell-1}$. Then the following fact holds:
  
  If $v=O\cdot \tilde{h}\binom{2}{1}$, then any V-fraction of type
  $O\cdot \tilde{h}\tilde{h}'\binom{2}{1}$ ($\tilde{h}'\in 
  \widetilde{\mathfrak{m}}$) is in $\mathcal{T}_{u_v}$, and conversely,
  any V-fraction from $\mathcal{T}_{u_v}$ is of type 
  $O\cdot \tilde{h}\tilde{h}'\binom{2}{1}$.

  Symmetrically, if $v=\tilde{g}\binom {1}{3}$, then any 
  $\tilde{g}\tilde{g}'\binom {1}{3}$ ($\tilde{g}'\in \widetilde{\mathfrak{m}}^*$) is in $\mathcal{T}_{u_v}$, and conversely,
  any V-fraction from $\mathcal{T}_{u_v}$ is of type $\tilde{g}\tilde{g}'\binom {1}{3}$. (Hence, the V-fractions of type $\tilde{g}'v$ need not be in $\mathcal{T}_{u_v}$!)

(b) Another tree transformation is the following. Fix a vertex $u$ of $\mathcal{T}$ and $g\in\mathfrak{m}^*$.
Then left multiplication by $g$ sends 
$\mathcal{T}_u$ isomorphically to $\mathcal{T}_{gu}$.
(Use that the left multiplication of $\mathfrak{m}^*$ commutes with the  right multiplication by $L$ and $R$, and this second multiplication generates the corresponding subtree.)

Using this one can also reprove the following statements of Theorem \ref{th:int3} from the introduction:
(i) 
the union of $\widetilde{\mathfrak{m}}^*$ orbits of elements $[\frac{1}{2}, \frac{2}{3}]\cap \mathbb{Q}$ equals $\mathcal{M}$;
(ii) 
the  $\widetilde{\mathfrak{m}}^*$ orbit of $\frac{1}{3}$ is $\mathcal{V}$;  (iii) 
furthermore, the orbits (i)-(ii) cover all rational numbers $\frac{q}{r}\in (0,1)$.



\end{remark}

\begin{remark}
  As we already observed, in $\SB_{2\ell}$ ($\ell>0$) there are $3^{\ell-1}$ newly created black intervals $\{I_{2\ell}^i\}_{i=1}^{3^{\ell-1}}$. Write $I_{2\ell}=\cup_i I_{2\ell}^i$ and let
  $\mu(\ell)$ be the length of $I_{2\ell}$. Then $s(\ell) \coloneqq \sum_{\ell'\leq \ell} \mu(\ell')$
  is the sum of lengths of all newly created black elementary intervals in the sequences $\SB_{\leq 2\ell}$.

Consider also the following construction. 
 Let $I^\circ_{2\ell}$ be the interior of $I_{2\ell}$
and write inductively $C_1 \coloneqq [0,1]\setminus I^\circ _{2}$, $C_\ell \coloneqq  C_{\ell-1}\setminus I^\circ _{2\ell}$
for $\ell>1$. Set  $C \coloneqq \cap _{\ell=1}^\infty C_\ell$.
Then   $\mathcal{V}= \{\frac{q}{r}\in C \cap (0, 1) \mid \ q,r \ \mbox{positive odd integers}\}$. In other words, 
$C\cap (0,1)\cap \mathbb{Q}$ consists of $\mathcal{V}$ and the boundaries of the newly created black intervals. 

We conjecture that $\lim_{\ell\to \infty}s(\ell)=1$, that is,
$C$ is a Cantor set (uncountable,
with Lebesgue measure zero,
nowhere dense,
perfect and closed,
totally disconnected).

\end{remark}

\section{Notation regarding lattices and Brieskorn singularities}\label{sec:brieskorn_notation}

\subsection{Notations regarding resolution of singularities}

Let $(X,0)$ be a complex analytic normal surface singularity 
and let $\pi: \widetilde{X}\to X$ be the minimal good resolution of $(X,0)$.
Let $E$ be the exceptional curve, and $E=\cup_{i=1}^s E_i$ its irreducible decomposition. The lattice 
of $\pi$ is $L = H_2(\tilde{X},\Z)$, it is a free $\Z$-module generated by the classes of $\{E_i\}_i$, and it is endowed with 
the negative definite intersection form $I=(E_i,E_j)_{i,j}$. 
We denote by $\Gamma$ the dual resolution graph, with vertices in bijection with $\{E_i\}_i$. 

In this note, \emph{we assume that the link $\Sigma$ of $(X,0)$ is an integral homology sphere}, that is, $\Gamma$ is a tree, all the irreducible exceptional curves are rational, and $I$ is unimodular. 

For any (full) subgraph $\Gamma'$ of $\Gamma$, let $I_{\Gamma'}=(\,,\,)_{\Gamma'}$ be the restriction of the intersection form 
to the sublattice generated by $E_a$'s corresponding to the vertices of $\Gamma'$. Then we write $\det (\Gamma') \coloneqq  
\det ( -I_{\Gamma'})$. It is a positive integer. 

 We also consider the dual basis $\{E^*_i\}_i \in L \cong L^*$ which satisfies for any $i,j$
\[
    (E^*_{i},E_j) = -\delta_{i,j},
\]
where $\delta_{i,j}$ is the Kronecker delta function. Note that we can write
\[
E^{*}_{i} = \sum_{j} - (I^{-1})_{i,j} E_j.
\]
In fact, all the entries $-(I^{-1})_{i,j}$ of $-I^{-1}$ are strictly positive. These entries are conveniently computed as 
\begin{equation}\label{eq:Iinverse}
    -(I^{-1})_{i,j}=\det (\Gamma- \overline{ij}),
\end{equation}
where $\Gamma- \overline{ij}$ is the subgraph obtained from $\Gamma$ by deleting the geodesic path connecting $i$ and $j$ and all adjacent edges. It can also be computed from 
 the splice diagram associated with $\Gamma$ (for splice diagrams see \cite{EisenbudNeumann}); it is the product of all the decorations adjacent to but not on the path $\overline{ij}$.
\subsection{Some topological invariants}\label{ss:top}
The $d$-invariant of the link $\Sigma$ can be viewed as an invariant of the lattice $L$ \cite{Nemethi2005, khan2026, os_floer_plumbed}:
\[
    d = \max \frac{k^2 +s}{4},
\]
where $s$ is the rank of $L$ and the maximum ranges over \emph{characteristic vectors} of $L$, namely vectors $k\in L$ that satisfy 
$
    (l,k) + (l,l) \equiv 0 \pmod{2}
$
for all $l \in L$. Since $k^2\leq 0$ we always have $d\leq s/4$.

Recall that in any case $d\geq 0$ 
and 
$L$ is diagonalizable over $\Z$ if and only if $d = 0$ 
\cite{OzsvathSzaboGraded,Elkies}. From a topological point of view, if $\Sigma $
is the boundary of a  rational ball then $d=0$.

There are a few distinguished characteristic vectors in $L$. The first one is the anti-canonical cycle, $Z_K \coloneqq -c_1(\Omega^2_{\widetilde{X}})\in L$, which can be determined by the adjunction relations $(Z_K,E_i)=E_i^2+2$ for all $i$. Rewritten in the dual basis it is 
\begin{equation}\label{eq:ZK}
Z_K=E+\sum_i (\delta_i-2)E^*_i,
\end{equation}
where $\delta_i$ is the number of adjacent vertices of $i$ in $\Gamma$. 

In fact, the set of characteristic vectors can be written as 
$Z_K+2L$. In this way we can define the second distinguished 
characteristic vector, the \emph{Wu class} of $\Gamma$ 
(or of the plumbed 4-manifold $\widetilde{X}$), as 
$w=\sum_i w_iE_i$, $w\equiv Z_K \ ({\rm mod}\ 2)$, and 
each $w_i\in\{0,1\}$. The \textit{Neumann--Siebenmann invariant} $\bar{\mu}$ is defined as follows \cite{Neumann_invariant}: 
\[
    \bar{\mu} = - w^2 - s.
\]
From definitions $-\bar{\mu}\leq 4d$. 

If $(X,0)$ is a hypersurface singularity
with geometric genus $p_g$ and signature of the Milnor fiber $\sigma $
then $\sigma+8p_g+Z_K^2+s=0$ \cite{laufer1977,Durfee1978}. If additionally the link is an integral homology sphere
(as in our case) then $\sigma \equiv 0 \ ({\rm mod}\ 8)$. Therefore, in such cases, 
$Z_K^2+s \equiv 0\ ({\rm mod} \ 8)$. Hence, 
for any characteristic class $k=Z_K+2l$ $(l\in L)$, one has $k^2+s\equiv 0\ ({\rm mod}\ 8)$ as well. In particular, $d$ is even and $\bar{\mu}$ is a multiple of 8.

There is a special class of lattices, the \textit{even lattices} (sometimes also called \textit{Type II lattices}) for which $d$ can be easily computed. By definition, a lattice is \textit{even} if 
$(x,x)$ is even for every $x\in L$. Note that if a lattice of 
a resolution is even then $0\in L$ is a characteristic vector,
hence $d=s/4$, hence it takes its upper bound. Moreover, for such lattices, $Z_K\in 2L$ and $w=0$. Hence $-\bar{\mu}=s=4d $
also holds. 

\subsection{Brieskorn hypersurface singularities}\label{ss:3.3}

We write $\Gamma=\Gamma(p,q,r)$ for the minimal good resolution graph of the Brieskorn singularity $x^p+y^q+z^r=0$, which is a plumbing graph of the Brieskorn sphere $\Sigma(p,q,r)$.
Let $L=L(p,q,r)$ be the corresponding lattice, and 
$d(p,q,r) = d(\Sigma(p,q,r))$ be its $d$-invariant. 

The graph $\Gamma$ consists of a (decorated) central vertex and 3 strings, which we call $p$-leg, $q$-leg and $r$-leg (see \cref{fig:brieskorn_plumbing_graph}).
The decorations on each leg can be obtained in three ways, which we describe next. 

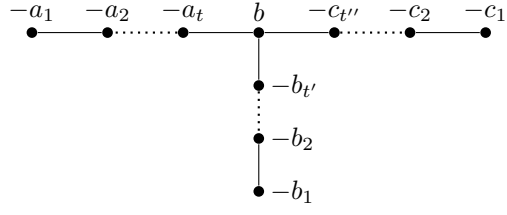
\begin{figure}
    \centering
    \input{graphics/plumbing/brieskorn_plumbing}
    \caption{Plumbing graph of $\Sigma(p, q, r)$.}
    \label{fig:brieskorn_plumbing_graph}
\end{figure}

In the first way, we consider the Hirzebruch--Jung (HJ) continued fraction 
\begin{equation}\label{eq:HJ}\frac{qr}{p}=[[a_0, a_1,a_2,\dots, a_t]] \coloneqq a_0 - \cfrac{1}{a_1-\cfrac{1}{\ddots - \cfrac{1}{a_t}}},
\end{equation}
where $a_0\geq 1$ and all other entries $a_i\geq 2$. 
Then the $p$-leg consists of $t$ vertices decorated by 
$-a_1, \ldots, -a_t$, with the vertex decorated by $-a_t$ connected to the central vertex by an edge. 
In the same way we obtain a continued fraction for the $q$-leg, denoted by $\{b_i\}_{i\geq 0}$ and for the $r$-leg, denoted by $\{c_i\}_{i\geq 0}$.

Note that the first step of the above continued fraction expansion can be written as 
\begin{equation}\label{eq:omega'}
\frac{qr}{p}=a_0-\frac{\omega'_1}{p},\end{equation}
where $0< \omega'_1<p$ and (from the identity (\ref{eq:omega'})) $qr\equiv -\omega'_1$ 
$({\rm mod}\ p)$. From the expansion of $qr/p$ we read that
$p/\omega'_1=
[[a_1, \ldots, a_t]]$. One defines similarly $\omega'_2$ and $\omega'_3$ for the other legs. 

 We also define integers 
 $\omega_1,\omega_2,\omega_3$ satisfying $0<\omega_1<p,\ 0<\omega_2<q, \ 0<\omega_3<r$ and
\begin{equation*}
    \omega_1 q r \equiv -1 \pmod{p},  \quad
    \omega_2 p r \equiv -1 \pmod{q},  \quad
    \omega_3 p q \equiv -1 \pmod{r}.
\end{equation*}
Then we have $\omega_1\omega'_1\equiv 1 \ ({\rm mod}\ p)$,
 $\omega_2\omega'_2\equiv 1 \ ({\rm mod}\ q)$,
  $\omega_3\omega'_3\equiv 1 \ ({\rm mod}\ r)$. They satisfy the 
  `symmetric' HJ continued fraction expansion
\[\frac{p}{\omega_1}=[[a_t,a_{t-1},\dots, a_1]]. \]
The decoration $b$ on the central vertex is given by
\begin{equation}\label{eq:b}
    b p q r +\omega_1 q r + \omega_2 p r + \omega_3 p q =- 1.
\end{equation}
The collection $(b; (p,\omega_1), (q,\omega_2), (r,\omega_3)) $
usually is called the \emph{set of Seifert invariants} of $\Sigma(p,q,r)$. 

In most of our computations we will use the expansion (\ref{eq:HJ}).
Though the entries $a_0,b_0,c_0$ do not appear on the resolution graph, they will play a key role
in our discussions.

If $\Gamma$ is the resolution graph of $\Sigma (p,q,r)$ then the entry $a_0$ of the above continued fraction expansion has two more interpretations.
First, $a_0=\det(\Gamma-v_{a_1})$
(where $\Gamma-v_{a_1}$ is obtained from $\Gamma$ by deleting the end-vertex $v_{a_1}$ decorated by $-a_1$).
Also, if we attach another vertex with
decoration $-a_0$ 
to the end vertex $v_{a_1}$ of the $p$ leg, then the determinant of this enlarged graph is zero. (Symmetric properties hold for the other legs as well.)

When $q=1$, which arises as an extremal case in later proofs, $\Sigma(p, 1, r)$ is diffeomorphic to $S^3$ and its minimal good resolution graph $\Gamma$ is empty. In this case we instead take $\Gamma$ to be the non-minimal resolution graph as in \cref{fig:brieskorn_plumbing_graph} which, in particular, has $b=-1$ and an empty $q$-leg.

\section{Stably equivalent lattices (via certain matrices of trace 2)}\label{s:Gs}

In what follows we will denote by $\langle -1\rangle$ the lattice of rank 1 generated by $e$ with $e^2=-1$.

\begin{definition}
We say that the negative definite lattices $L$ and $L'$ are \emph{stably 
equivalent} if there exist $N,N'\geq 0$ such that 
$L\oplus \langle -1\rangle ^{\oplus N}= L'\oplus \langle -1\rangle ^{\oplus N'}$.
\end{definition}

\begin{lemma}
    If $L$ and $L'$ are stably equivalent, then their $d$-invariants are equal.
\end{lemma}

\begin{example}
Let $\Gamma$ be a negative definite graph and $v$ one of its vertices.
Assume that $\Gamma'$ is obtained from $\Gamma$ by blowing up the vertex $v$. Denote the corresponding lattices by $L$ and $L'$ and the newly created generator in $L'$ by $E_{new}$. Then 
$L \oplus \langle -1\rangle\stackrel{\simeq}{\longrightarrow}L'$. Indeed, we identify the base elements associated with vertices naturally in $L$ and $L'$, $E_u\rightsquigarrow E_u'$. Then the map $E_u\mapsto E_u'$ for $u\not=v$, $E_v\mapsto E_v'+E_{new}$ and $e\mapsto E_{new}$
realizes the lattice isomorphism.
\end{example}
\begin{example}\label{ex:2}
Assume that $\Gamma$ is a connected negative definite graph with $\det(\Gamma)=1$. Let $M$ be the associated plumbed 3-manifold. 
Choose an end-vertex $v$ and let $K$ be the corresponding knot in $M$
(it is the generic $S^1$-fiber associated with $v$ in the plumbing construction of $M$).
Then consider the surgery 3-manifold $M' \coloneqq  M_{-1}(K)$. The 
plumbing graph $\Gamma'$ of $M'$ can be obtained from $\Gamma$ as follows (see \cref{example_three_schematic}): attach a new vertex $v_{new}$ with decoration 
$-\Delta-1$ to $v$ by an edge, where $\Delta \coloneqq \det(\Gamma-v)$.

Then the lattices $L$ and $L'$ of $\Gamma$ and $\Gamma'$ satisfy
$L \oplus \langle -1\rangle \stackrel{\simeq}{\longrightarrow}L'$.
(See \cite{Lidman_2018}.) Indeed, identify naturally the base elements $E_u$ with generators 
$E'_u$ of $L'$.
Note that $\{\{E_u'\}_u, E^*_{new}\}$ forms a basis 
in $L'$ since the $E_{new}$-coefficient of $E_{new}^*$ is $-(E_{new}^*, E_{new}^*)=\det(\Gamma)=1$. The lattice isomorphism is realized by 
$E_u\mapsto E'_u$, $e\mapsto E^*_{new}$.
\end{example}

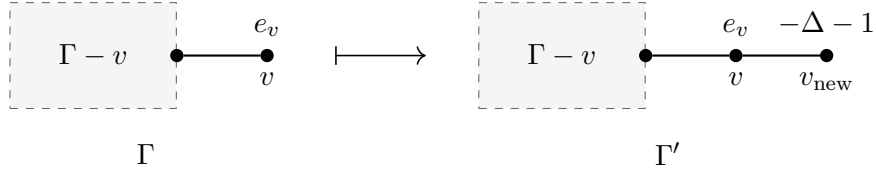
\begin{figure}
    \centering
    \input{graphics/plumbing/example_three}
    \caption{Schematic of the resulting $\Gamma'$ plumbing graph from $\Gamma$ in Example \ref{ex:2}. In the above $\Delta = \det(\Gamma - v)$. The corresponding lattices $L$ and $L'$ are stably equivalent.}
    \label{example_three_schematic}
\end{figure}

Next we analyse Brieskorn spheres of type $\Sigma(2,q,r)$ with $0<q<r$. 
We will consider different operations connecting them. 
In what follows 
we will often use the notation $\mat{a}{b}{c}{d}$ for 
matrices $\begin{inlinematrix}
    a & b\\
    c & d
\end{inlinematrix}\in{\rm GL}(2,\Z)$ whose entries satisfy $a d - b c =\pm 1$. 
These operators will be used as special transformations of the Brieskorn exponents $(q,r)$:
\[
\begin{pmatrix}
    q\\
    r
\end{pmatrix}
\ \mapsto \ 
\begin{pmatrix}
    q'\\
    r'
\end{pmatrix}
 \coloneqq  \Mat{a}{b}{c}{d}
\begin{pmatrix}
    q \\
    r
\end{pmatrix}.
\]

\begin{example}
Consider the matrix $h_1 \coloneqq (1,0;2,1)$. 
Then $q'=q$ and $r'=2q+r$. Consider the HJ continued fractions 
$qr/2=[[a_0,2]]$, $2r/q=[[b_0, b_1,\ldots]]$ and
$2q/r=[[c_0,c_1,\ldots]]$. Then the corresponding continued fractions for $(2,q',r')$ are the following 
$q'r'/2=[[q^2+a_0,2]]$, $2r'/q'=[[4+b_0, b_1,\ldots]]$ and
$2q'/r'=[[1, 1+c_0,c_1,\ldots]]$. Let $v$ be the end-vertex of the $r$-leg in $\Gamma$. 
Then the legs of the graph $\Gamma'$ of $\Sigma(2,q',r')$ are obtained from the legs of the old one $\Gamma$ by gluing a new vertex $v_{new}$ to $v$ with decoration 
 $-(c_0+1)$. Recall that $c_0=\det (\Gamma-v)$.

 Now, let $K$ be the knot associated with $v$ 
 (e.g. the Seifert fiber associated with $v$). 
Then we can consider the surgery Seifert 3-manifold $\Sigma(2,q,r)_{-1}(K)$ as well. By Example \ref{ex:2}
and the above continued fraction argument, the legs of the graph of $\Sigma(2,q,r)_{-1}(K)$ and of $\Gamma'$ agree. Since both of them have determinant 1, the decoration of their central vertex should also agree. Hence 
$\Sigma(2,q', r')=\Sigma(2,q,r)_{-1}(K)$ and 
$L'\simeq L\oplus \langle -1\rangle$ via  Example \ref{ex:2}.
\end{example}

\begin{example}
Consider the matrix $h_2 \coloneqq (0,1;1,2)$. 
Then $q'=r$ and $r'=q+2r$.
This is a `symmetric' case of the previous example, 
 $\Sigma(2,q', r')=\Sigma(2,q,r)_{-1}(K)$, where 
$K$ is the Seifert fiber of the $q$-leg. If $q=1$ then $K$ is understood to be the generic fiber of the central vertex (for which the construction in \cref{ex:2} also holds).
In particular, 
$L'\simeq L\oplus \langle -1\rangle$ again. 
\end{example}

\begin{proposition}\label{ex:5}
Consider the matrix $h_4 \coloneqq \Mat{-1}{2}{-2}{3}$. 
Then $L'$ and $L$ are stably equivalent.
\end{proposition}
\begin{proof}
In this case  $q'=-q+2r$ and $r'=-2q+3r$. Fix the HJ continued fractions 
$qr/2=[[a_0,2]]$, $2r/q=[[b_0, b_1,\ldots]]$ and
$2q/r=[[c_0,c_1,\ldots]]$. Since $2r/q>2$ we have $b_0\geq 3$.
Similarly, $2q/r<2$ hence $1\leq c_0\leq 2$. 
Then the corresponding continued fractions for $(2,q',r')$ are the following: 
$q'r'/2=[[q^2-4qr+3r^2+a_0,2]]$, $2r'/q'=[[b_0', b'_1,\ldots ]]=
[[3,b_0-1, b_1,\ldots]]$ and
$2q'/r'=[[c_0', c'_1,\ldots ]]=[[2, 1, c_0-2 ,c_1,\ldots]]$.
Hence the $q$ leg of $\Gamma$ is modified into the $q'$ leg of $\Gamma'$ as follows:

\begin{tikzpicture}[
  on grid,
  node distance=1.2cm,
  dot/.style={circle,fill,inner sep=0pt,minimum size=4pt},
  label distance=1pt,
  stub/.style={xshift=10pt},
  ldots/.style={right=0.7cm of #1},
]

\node[dot,label=above:$-b_1$]             (a1) {};
\node[dot,right=of a1,label=above:$-b_2$] (a2) {};
\node[ldots=a2]                           (ad) {$\ldots$};
\draw (a1) -- ([stub]a2.center);

\node[right=of ad] (arr) {$\longrightarrow$};

\node[dot,right=of arr,label=above:$-b_0+1$] (b0) {};
\node[dot,right=of b0, label=above:$-b_1$]   (b1) {};
\node[dot,right=of b1, label=above:$-b_2$]   (b2) {};
\node[ldots=b2]                              (bd) {$\ldots$};
\draw (b0) -- ([stub]b2.center);

\end{tikzpicture}

Since $b_0\geq 3$ the decoration $-b_0+1\leq -2 $, hence this is a `correct' negative definite graph decoration.

Regarding the $r$-leg modification, the $r'$ leg of $\Gamma'$ is the following
(the third diagram is obtained from the second one by blow down of the $(-1)$ vertex):

\begin{tikzpicture}[
  on grid,
  node distance=1.2cm,
  dot/.style={circle,fill,inner sep=0pt,minimum size=4pt},
  label distance=1pt,
  stub/.style={xshift=10pt},
  ldots/.style={right=0.7cm of #1},
]

\node[dot,label=above:$-c_1$]             (a1) {};
\node[dot,right=of a1,label=above:$-c_2$] (a2) {};
\node[ldots=a2]                           (ad) {$\ldots$};
\draw (a1) -- ([stub]a2.center);

\node[right=of ad] (arr) {$\longrightarrow$};

\node[dot,right=of arr,label=above:$-1$]     (b0) {};
\node[dot,right=of b0, label=above:$-c_0+2$] (b1) {};
\node[dot,right=of b1, label=above:$-c_1$]   (b2) {};
\node[dot,right=of b2, label=above:$-c_2$]   (b3) {};
\node[ldots=b3]                              (bd) {$\ldots$};
\draw (b0) -- ([stub]b3.center);

\node[right=of bd] (eq) {$\cong$};

\node[dot,right=of eq,label=above:$-c_0+3$] (d0) {};
\node[dot,right=of d0,label=above:$-c_1$]   (d1) {};
\node[dot,right=of d1,label=above:$-c_2$]   (d2) {};
\node[ldots=d2]                             (dd) {$\ldots$};
\draw (d0) -- ([stub]d2.center);

\end{tikzpicture}

Now consider the graph $\Gamma_{+1,+3}$ 
of the 3-manifold obtained from $\Sigma(2,q,r)$ by 
$\mu_q=1$ surgery along the $q$-leg, and $\mu_r=3$ surgery along the $r$-leg. From the above discussion we obtain that the legs of $\Gamma'$ and of $\Gamma_{+1,+3}$ agree.
By \cite[Prop. 5.3.11]{Gompf1999} (which holds for surgeries in homology spheres, as well as surgeries in $S^3$), the lattice of $\Gamma_{+1,+3}$ is unimodular if and only if 
the matrix $(\mu_q,p;p,\mu_r)$ is unimodular. Since this is the case,   both graphs
$\Gamma_{+1,+3}$ and $\Gamma'$ are unimodular, and their legs coincide. Therefore,  their central decorations should agree as well. Hence $\Gamma'$ is obtained from $\Gamma$ by  the 
corresponding leg modifications described above. 
    
Notice also that since 
$c_0\in\{1,2\}$, $-c_0+3$ is a positive number,  we have to modify this leg by (orientation preserving) 
plumbing calculus in order to get the genuine 
negative definite  plumbing graph $\Gamma'$ (see below).

\vspace{2mm}

\noindent \textbf{Case A.} \ $c_0=2$. 

In this case $-c_0+3=+1$, hence we can blow down this $(+1)$
vertex, which transforms the adjacent decoration $-c_1$ to $-c_1-1$. In this way we get a negative definite graph $\Gamma'$;
it is obtained from $\Gamma$ by two modifications: we glue to the end of the $q$ leg a new vertex with decoration $-b_0+1$ and we modify the 
end vertex decoration of the $r$ leg from $-c_1$ into $-c_1-1$. 
Note that  ${\rm rank}(L')={\rm rank}(L)+1$.

Since in $\Gamma$, the end vertex decoration of the $r$ leg was modified, the lattice $L$ does not embed (naturally) into $L'$. Let us replace 
$\Gamma$ (hence $L$ as well) by its blow up of the  end vertex
of its $r$ leg,  and denote it by $\Gamma^\sharp$ (respectively $L^\sharp$). The (schematic) graphs $\Gamma^\sharp$ and $\Gamma'$ are the following:

\begin{tikzpicture}[
  on grid,
  node distance=1.05cm,
  dot/.style={circle,fill,inner sep=0pt,minimum size=4pt},
  label distance=1pt,
]

\node[dot,label=above:$-2$]          (a1) {};
\node[dot,right=of a1]               (a2) {};
\node[right=of a2]                   (ad) {$\ldots$};
\node[below=10pt of a2]                   {$\vdots$};
\node[dot,below=of a2,
      label=right:$-b_1$,label=left:$G$] (G)  {};

\node[dot,right=1.4cm of ad,label=above:$-c_2$]                     (c2) {};
\node[dot,right=1.4cm of c2,label=above:$-c_1-1$,label=below:$F$]   (F)  {};
\node[dot,right=1.4cm of F, label=above:$-1$,    label=below:$E$]   (E)  {};

\draw (a1) -- ([xshift=10pt]a2.center);
\draw (G)  -- ++(0,8pt);
\draw ([xshift=-10pt]c2.center) -- (E);

\node[dot,right=2.5cm of E,label=above:$-2$] (b1) {};
\node[dot,right=of b1]                       (b2) {};
\node[right=of b2]                           (bd) {$\ldots$};
\node[below=10pt of b2]                           {$\vdots$};
\node[dot,below=of b2,
      label=right:$-b_1$,label=left:$G'$]    (G') {};
\node[dot,below=of G',
      label=right:$-b_0+1$,label=left:$D'$]  (D') {};

\node[dot,right=1.4cm of bd,label=above:$-c_2$]                      (c2') {};
\node[dot,right=1.4cm of c2',label=above:$-c_1-1$,label=below:$F'$]  (F')  {};

\draw (b1) -- ([xshift=10pt]b2.center);
\draw (D') -- ([yshift=8pt]G'.center);
\draw ([xshift=-10pt]c2'.center) -- (F');

\end{tikzpicture}

It is helpful to  establish some identities in $L'$:
\begin{enumerate}[label=(\roman*)]
\item $-(F'^*, D'^*)_{L'}$ is 
the determinant of the subgraph obtained from $\Gamma'$ by deleting the path connecting (the vertex) $F'$ and $D'$, hence it is $c_0'=2$. 

\item $-(D'^*, D'^*)_{L'}$ is the first term $b_0'$ in the expansion of
$2r'/q'$, hence it is 3.

\item We claim that $-(F'^*, F'^*)_{L'}=2$. 
\end{enumerate}
To prove this first recall that $2q'/r'=[[2,1,c_0-2, c_1, \ldots]]$, 
and, since $c_0=2$,  it is in fact $[[2,1,0, c_1,\ldots]]$.  
 Then, we know by the interpretation of $c'_0$ that the 
determinant of the graph ($\Gamma'$
before the plumbing calculus enlarged by $-c_0'=-2$, cf. \cref{sec:brieskorn_notation}) whose 
$r'$ leg  ends as \tikz[inlinediag]{
  \node[vertex, label={above:$-c_1$}] (u1) at (0,0)   {};
  \node[vertex, label={above:$0$}]    (u2) at (1.6,0) {};
  \node[vertex, label={above:$-1$}]   (u3) at (3.2,0) {};
  \node[vertex, label={above:$-2$}]   (u4) at (4.8,0) {};
  \draw (u1) -- (u2) -- (u3) -- (u4);
  \draw[dotted] (u1) -- (-1.2,0);
}
has determinant zero. After 0-absorption, we get that the graph 
\tikz[inlinediag]{
  \node[vertex, label={above:$-c_1-1$}] (u1) at (0,0) {};
  \node[vertex, label={above:$-2$}]     (u2) at (2,0) {};
  \draw (u1) -- (u2);
  \draw[dotted] (u1) -- (-1.2,0);
}
has determinant zero as well. This is equivalent (by expanding the determinant) to the statement that the determinant of the matrix obtained from this by deleting the vertices with decoration $-c_1-1$ and $-2$ is 2. But this subgraph is exactly the subgraph obtained from $\Gamma'$
by deleting the vertex corresponding to $F'$, hence its determinant equals $-(F'^*, F'^*)_{L'}$.

\vspace{2mm}

We claim that $L^\sharp\simeq L'$.
Note that $\Gamma^\sharp-\mbox{vertex}(E)$ and $\Gamma'-\mbox{vertex}(D')$ (and their lattices too) can be identified. We construct 
$\alpha :L^\sharp\to  L'$ as follows. The restriction of $\alpha$ to 
$L(\Gamma^\sharp-\mbox{vertex}(E))$ is the identity (by the above identification).
On the other hand, $\alpha(E) \coloneqq  -F'^*+D'^*$.

We need to verify a few identities:
\begin{itemize}
 \item \
$(\alpha(F),\alpha(E))_{L'}=(F', -F'^*+D'^*)_{L'}=1=(F,E)_{L^\sharp},$

\item \
$(\alpha(H),\alpha(E))_{L'}=(H', -F'^*+D'^*)_{L'}=0=(H,E)_{L^\sharp}$ for any base vector in $L^\sharp$ different than $E$, $F$,

\item \
$(\alpha(E),\alpha(E))_{L'}=( -F'^*+D'^*,  -F'^*+D'^*)_{L'}$, which,  via (i)-(ii)-(iii) equals 
$-2-3+4=-1$, hence it equals $(E,E)_{L^\sharp}$. 
\end{itemize}
Finally we have to verify that $L(\Gamma'-{\mbox{vertex}(D')})+ \Z\langle  -F'^*+D'^*
\rangle=L'$. But this follows from the identity $(  -F'^*+D'^*, -D'^*)=
-2+3=1$.

This ends the proof of $L^\sharp\simeq L'$. In particular, 
$L\oplus \langle-1\rangle \simeq L'$ as well.
\vspace{2mm}

\noindent \textbf{Case B.} \ $c_0=1$.  

The proof is very similar to the proof of Case A, we modify slightly those arguments.

In this case $-c_0+3=+2$. Then we have to blow up the edge connecting the vertices decorated by $-c_1$ and $+2$, then blow down the end $(+1)$-vertex. 
In this way we get that the legs of the negative definite graph $\Gamma'$
are obtained from the legs of $\Gamma$ by three modifications: we glue to the end of the $q$ leg\footnote{If $q=1$, we glue to the central vertex instead, since the $q$-leg is empty.} a new vertex with decoration $-b_0+1$ and we modify the 
end vertex decoration of the $r$ leg from $-c_1$ into $-c_1-1$ and we glue to it another  $(-2)$-vertex.  By the same argument as in the Case A we obtain that in fact 
this graph is $\Gamma'$ (i.e. the central decoration is also preserved). 
Note that  ${\rm rank}(L')={\rm rank}(L)+2$. 

Next, we modify 
$\Gamma$ by two blow ups: first we blow up the  end vertex
of its $r$ leg, and then the newly created vertex.  This new graph will be denoted by 
 $\Gamma^{\sharp\sharp}$, and  its lattice by  $L^{\sharp\sharp}$. 
 The (schematic) graphs $\Gamma^{\sharp\sharp}$ and $\Gamma'$ are the following:

\begin{picture}(300,100)(20,0)
 \put(30,70){\circle*{4}}
    \put(60,70){\circle*{4}}
\put(30,70){\line(1,0){40}}
 \put(90,70){\makebox(0,0){$\ldots$}}
        \put(30,80){\makebox(0,0){$-2$}}
       \put(60,60){\makebox(0,0){$\vdots$}}
          \put(60,40){\circle*{4}}
             \put(73,40){\makebox(0,0){$-b_1$}}
      
 \put(130,70){\circle*{4}} \put(250,70){\circle*{4}}
 \put(170,70){\circle*{4}}
 \put(210,70){\circle*{4}}
 \put(120,70){\line(1,0){130}}
  \put(130,80){\makebox(0,0){$-c_2$}} 
  \put(170,80){\makebox(0,0){$-c_1-1$}}  
  \put(210,80){\makebox(0,0){$-2$}}\put(250,80){\makebox(0,0){$-1$}}
\put(60,40){\line(0,1){8}}
 
  \put(210,60){\makebox(0,0){$F$}}
    \put(250,60){\makebox(0,0){$E$}}
      \put(50,40){\makebox(0,0){$G$}}

 \put(280,70){\circle*{4}}
    \put(310,70){\circle*{4}}
\put(280,70){\line(1,0){40}}
 \put(340,70){\makebox(0,0){$\ldots$}}
        \put(280,80){\makebox(0,0){$-2$}}
       \put(310,60){\makebox(0,0){$\vdots$}}
          \put(310,40){\circle*{4}}
             \put(323,40){\makebox(0,0){$-b_1$}}
      
 \put(380,70){\circle*{4}} \put(460,70){\circle*{4}}
 \put(420,70){\circle*{4}}
\put(310,10){\circle*{4}}
 \put(370,70){\line(1,0){90}} \put(310,10){\line(0,1){38}}
  \put(380,80){\makebox(0,0){$-c_2$}}   \put(460,80){\makebox(0,0){$-2$}} 
  \put(420,80){\makebox(0,0){$-c_1-1$}}  
 \put(333,10){\makebox(0,0){$-b_0+1$}}
 
  \put(460,60){\makebox(0,0){$F'$}}
    \put(300,10){\makebox(0,0){$D'$}}
      \put(300,40){\makebox(0,0){$G'$}}
  
\end{picture}

The following identities in $L'$  hold by similar arguments: \

(i) $-(F'^*, D'^*)_{L'}=2$, 
%
(ii) $-(D'^*, D'^*)_{L'}=3$, 
%
(iii)  $-(F'^*, F'^*)_{L'}=2$. 
%
%


We claim that $L^{\sharp\sharp}\simeq L'$.
Similarly as above 
we identify $\Gamma^{\sharp\sharp}-\mbox{vertex}(E)$ and $\Gamma'-\mbox{vertex}(D')$ and  we construct 
$\alpha :L^{\sharp\sharp} \to  L'$ as follows. The restriction of $\alpha$ to 
$L(\Gamma^{\sharp\sharp}-\mbox{vertex}(E))$ is the identity  and  $\alpha(E) \coloneqq  -F'^*+D'^*$.
All the remaining verifications are analogous to the previous case. Therefore, 
$L\oplus \langle-1\rangle^{\oplus 2} \simeq L'$ holds as well.
\end{proof}

\section{Equivalent lattices (via a finite order matrix)}\label{sec:equivalent_lattices}

In this section we consider the general case of Brieskorn hypersurface singularities $x^p+y^q+z^r=0$ with $p<q<r$ such that the link $\Sigma(p,q,r)$ is an integral homology sphere,
equivalently, the integers $p,q,r$ are pairwise relative prime. 

We will fix the integer $p$ and we will consider a linear operator $\binom{q}{r}\mapsto \binom{q'}{r'}$ and we will show that the lattices associated with $(p,q,r) $ and $(p,q',r')$ are isomorphic.  

\subsection{The reflection transformation $S_p$.}\label{ss:Ep}

For any fixed $p\geq 1$ consider the matrix $S_p\in {\rm GL}(2,\Z)$
\[S_p \coloneqq  \begin{pmatrix} -2p & 2p-1 \\ -2p-1 & 2p
\end{pmatrix}.\]
One can check that $S_p^2$ is the identity matrix, $\det(S_p)=-1$ and the characteristic polynomial is $t^2-1$. 
We will analyse the (left) action of $S_p$, $\binom{q}{r}\mapsto \binom{q'}{r'} \coloneqq S_p\binom{q}{r}$ on the interval
$\frac{p-1}{p}\leq \frac{q}{r}\leq \frac{p}{p+1}$.
The operator $S_p$  preserves the interval $[\frac{p-1}{p}, \frac{p}{p+1}]$ (but reverses the orientation), it
interchanges the two boundary numbers $\frac{p-1}{p}$ and $\frac{p}{p+1}$, and keeps fixed 
their Stern--Brocot mediant $\frac{2p-1}{2p+1}$.

\begin{proposition}\label{prop:EP}
The lattices associated with $(p,q,r) $ and $(p,q',r')$ are isomorphic.  
\end{proposition}
\begin{proof}
We fix $(p,q,r)$ with 
$\frac{p-1}{p} <\frac{q}{r} < \frac{p}{p+1}$ and set
    $\binom{q'}{r'} \coloneqq S_p\binom{q}{r}$.

\vspace{2mm}

\noindent \textbf{Step 1.} \ A computation shows that 
 \[
    \frac{pq}{r} = p - \cfrac{1}{2-\cfrac{1}{2-\cfrac{1}{\frac{(2p-1) r-2pq}{(3p-2) r-3p q}}}}
\ \ \mbox{and} \ \ 
  \frac{pr'}{q'} = p +2 - \cfrac{1}
  {\frac{(2p-1) r-2pq}{(3p-2) r-3p q}}.
    \]
Therefore, the continued fractions of $\frac{pq}{r}$ and 
$\frac{pr'}{q'}$ have the following form
\[\frac{pq}{r}= [[p, 2,2,c_3, c_4, \ldots ]] \ \ \mbox{and} \ \ 
\frac{pr'}{q'}= [[p+ 2,c_3, c_4, \ldots ]].\]
Symmetrically, 
\[\frac{pr}{q}= [[p+ 2,b_1, b_2, \ldots ]] \ \ \mbox{and} \ \ 
\frac{pq'}{r'}= [[p, 2,2,b_1,b_2, \ldots ]].\]
Finally, note that 
\begin{equation}\label{eq:A}\frac{q'r'}{p}=A+\frac{qr}{p}, \ \ \mbox{where} \ \ A=4p(r-q)^2-2(r^2-q^2),\end{equation}
hence the $p$-legs of the two graphs are the same.

\vspace{2mm}

\noindent \textbf{Step 2.} \
The following schematic graph is the minimal good plumbing graph $\Gamma$ of $\Sigma(p,q,r)$;   
 we mark three vertices $i,j,n$, where 
 $E_i,E_j,E_n$ will denote the corresponding generators in the lattice. The two $(-2)$-vertices sit at the end of the $r$-leg. 
    \begin{center} 
    \input{graphics/plumbing/plumbing_before.tex}
    \end{center}
  We consider also the corresponding (maximal) splice diagram
    \begin{center}
    \input{graphics/plumbing/splice_max.tex}
    \end{center}

    We consider the vector 
    \[F = E^*_i - E^*_j - E^*_n \in L.
    \]
    We claim that the $E_i$-coefficient of $F$ is $-1$. Indeed, this coefficient is $-(F,E_i^*)=-( E^*_i - E^*_j - E^*_n, E_i^*)= -(I^{-1})_{ii}+(I^{-1})_{ji}+(I^{-1})_{ni}$,
     which by splice diagram computation   
     equals $2y-y-2p=y-2p$. By the interpretation of the first coefficient 
     $\lceil \frac{pq}{r}\rceil$ of the continued fraction of 
     $\frac{pq}{r}$ we get that $p=z$. Moreover, since the determinant of $\Gamma$ is 1, expanding the line of $E_j$ we get
    $2z-y=1$. Hence $y=2p-1$. Therefore,  $-(F,E_i^*)=y-2p=-1$.

    This shows that the collection of elements 
    $F, E_j,E_n, \{E_m\}_{m\not=i,j,n}$
    forms basis of $L$. 
    The vector $F$  vector is orthogonal to all $E_m$ with $m$ different from $i,j,n$. It has intersection number 1 with $E_j$, as well as with $E_n$. This means that in the basis 
    $F, E_j,E_n, \{E_m\}_{m\not=i,j,n}$ the intersection 
    matrix can be represented by the following graph,
    once we prove that $(F,F)=-2$ (the self-intersections of the unmarked vertices are evidently preserved).
    \begin{center}

\input{graphics/plumbing/plumbing_after.tex}
    \end{center}
    The verification of $(F,F) = -2$ uses again a splice computation 
    \begin{align*}
   - (E^*_i - E^*_j - &E^*_n,E^*_i - E^*_j - E^*_n) \\ &= 
   -\big((I^{-1})_{ii}+ (I^{-1})_{jj} + (I^{-1})_{nn} 
    - 2 (I^{-1})_{ij} - 2 (I^{-1})_{in} + 2 (I^{-1})_{jn} \big)\\
    &= 2y +z +x - 2y - 4p + 2p= z+x-2p.
    \end{align*}
   But $z=p$ and by similar argument, $x=\lceil \frac{pr}{q}\rceil =p+2$, hence 
$(F,F)=-2$. 

\vspace{2mm}

\noindent \textbf{Step 3.} \
Now, we compare this second graph with the graph of $\Sigma(p,q',r')$.
By Step 1  we obtain that their decorated legs can be identified. Moreover, both have determinant 1, hence the self-intersections of the central vertices should agree as well. 
Hence the two graphs (and their lattices  too) are the same. 
(That is, the minimal good 
plumbing graph of 
$\Sigma(p,q',r')$ is obtained from the graph of $\Sigma(p,q,r)$ by removing two $(-2)$-vertices from the end of the $r$-leg and attaching them to the $q$-leg.)
Furthermore, their lattices are isomorphic. 
\end{proof}

\begin{remark} There are several major conceptual differences between the operators $\{h_i\}_{i=1,2,4}$ of \cref{s:Gs} and the operator $S_p$.
The most apparent one is that $h_1,h_2,h_4$ generate a free monoid in ${\rm GL}(2,\Z)$ and $S_p$ has order two. Another is topological: the operators $h_i$ usually do not preserve the $\bar{\mu}$ invariant. However, $S_p$ does. 
\end{remark}

The next proposition is an addendum of Proposition \ref{prop:EP}.

\begin{proposition}\label{prop:mubar}
The Neumann--Siebenmann invariant is preserved under the action of $S_p$: 
\begin{equation}
    \bar{\mu}(p,q,r) = \bar{\mu}(p,q',r').
\end{equation}
\end{proposition}
\begin{proof} Consider the notations of the proof of Proposition \ref{prop:EP}. 
We compare the Wu classes of $\Gamma$ and $\Gamma'$. 
In the case of the Wu class of $\Gamma$, from $(w,E_j)_\Gamma\equiv 0 \ ({\rm mod} \ 2)$ we obtain 
 $w_i=0$. Symmetrically, the same vanishing holds for the component of the Wu class of $\Gamma'$ associated to the vertex $F$.

Next, we compare the component associated with the end-vertices of the Wu classes of $\Gamma$ and $\Gamma'$. We provide the details for the end $E_j$ of $\Gamma$. 

 First, recall that $w(\Gamma)\equiv Z_K(\Gamma) \ ({\rm mod} \ 2)$, hence we will focus on the 
 $E_j$-coefficient of $Z_K=Z_K(\Gamma)$. Using (\ref{eq:ZK}) we have 
 $-(Z_K, E^*_j)_{\Gamma}=-(E+\sum _m (\delta_m -2) E^*_m, E^*_j)= 1+pq-p-q-c_0 =
 (p-1)(q-1)-\lceil \frac{pq}{r}\rceil$. Since $(p-1)(q-1)$ is even, from the continued fraction expansions from Step 1 of the proof of Proposition \ref{prop:EP}
 we get $w_j (\Gamma)\equiv p \ ({\rm mod}\ 2)$. For the $q$ leg of $\Gamma$ and the 
 $q'$ and $r'$ legs of $\Gamma'$ the computations are the same. 
 From Step 1 we get that all of them are congruent mod 2 by $p$. 
 Moreover, if in $\Gamma$ we write  $(w(\Gamma), E_i)_\Gamma\equiv  w_j+w_k\equiv 0 \ ({\rm mod}\ 2)$, then we get that the $w_k\equiv p\ ({\rm mod} \ 2)$ too. 
 Similar congruence holds for the $E_n$-coefficient of $w(\Gamma')$.
This shows that when we cut the two $(-2)$-vertices from $\Gamma$ and we glue them to 
$\Gamma'$, we will not effect the  restrictions of the Wu classes on 
 $\Gamma- {\rm vertex}(E_i,E_j)$ and $\Gamma'- {\rm vertex}(F,E_j)$, hence they can be identified.    
\end{proof}
\begin{remark} We checked on many examples that $S_p$ also preserves the associated monotone graded subroots \cite{Dai_Manolescu_2019}. On the other hand,  there are several other topological invariants which are not preserved by the operator $S_p$. As an example, consider the following computation of the Casson invariant for $\Sigma(p, q, r)$ and $\Sigma(p', q', r')$. 

For an integral homology sphere which is a plumbed 3-manifold $\Sigma(\Gamma)$ and represented by a graph $\Gamma$, the Casson invariant $\lambda(\Sigma)$ can be computed as (c.f. \cite[Section 5]{Nemethi_2002_SW}): 
    \begin{equation}\label{eq:casson}
    -24 \lambda(\Sigma) = \sum_i ({E_i^2} +3) + \sum_i (2-\delta_i)  (I^{-1})_{ii}.
\end{equation}
In our case, the expressions $ \sum_i ({E_i^2} +3)$ for  $\Gamma$ and $\Gamma'$ agree. On the other hand, $\sum_i (2-\delta_i)  (I^{-1})_{ii}$ 
for  $\Gamma$ is $pqr-a_0-b_0-c_0=pqr-\lceil \frac{qr}{p}\rceil- \lceil \frac{pr}{q}\rceil
-\lceil \frac{pq}{r}\rceil= pqr-(p+p+2+a_0)$ (use Step 1 from the proof of Proposition \ref{prop:EP}). Similarly, the same expression for
$\Gamma'$ is $pq'r'-(p+p+2+A+a_0)$. Hence $-24 \lambda (\Sigma')+24 \lambda(\Sigma)=
pq'r'-pqr-A=(p^2-1)A$, where $A$ is given in (\ref{eq:A}).


Recall also that for integral homology Brieskorn singularities, $8\lambda$ equals the 
signature of the Milnor fiber. Further computations show that 
the operator $S_p $ does not preserve the 
Milnor number, the signature, the geometric genus, or $Z_K^2+s$. 
\end{remark}

\begin{remark}
    The transformation $S_2$ was used by Yamada and Matsumoto to prove that if the minimal resolution graph $\Sigma(2,q,r)$ is even, then $1/2<q/r<2/3$. See \cite[Theorem 1.2]{Yamada_Matsumoto_1983}.
\end{remark}

\subsection{Surgery coefficients}\label{ss:SURG}
Though in the previous discussions we used the surgery notations too, for the convenience of the reader we 
make explicit in the case of $\Sigma(p,q,r)$ the correspondence between the entries of 
the $2\times 2$ matrix $T$ with $T\binom{q}{r}=\binom{q'}{r'}$ and the surgery coefficients.

Recall that we perform surgeries along 
the singular Seifert fibers $K_q$ and $K_r$
associated with the $q$ and $r$ legs of $\Sigma(p,q,r)$.
The output surgery manifold 
 $\Sigma(p,q,r)_{s_q/t_q,s_r/t_r}(K_q,K_r)$ is obtained 
 by $s_q/t_q $ surgery along $K_q$ and 
 $s_r/t_r $ surgery along $K_r$.
Since we use only negative definite plumbing graphs, we prefer to 
rewrite the surgery coefficients as :
$s_q/t_q=-x_q/y_q$, $s_r/t_r=-x_r/y_r$.

As usual, we allow also $\infty$ surgery too, e.g., if
$-x_q/y_q=\infty$ then the surgery along $K_q$ is trivial, and 
$q'=q$ (similarly for the $r$ leg).
Moreover, if the coefficients are not $\infty$ then we assume 
$y_q>0$ (resp. $y_r>0$). 

Recall that the Seifert invariants of the $r$ leg are $(r,\omega'_3)$, cf. subsection (\ref{ss:3.3}).

 If we perform a $-x_r/y_r$ surgery along $K_r$ then the $r$ leg 
 of $\Sigma(p,q,r) $ is modified as follows (the vertex decorated by $-c_{t''}$ is glued to the central vertex in both cases):

\begin{picture}(300,40)(0,-10)
 \put(30,5){\circle*{4}}
    \put(60,5){\circle*{4}}
 \put(110,5){\line(-1,0){10}}
     \put(110,5){\circle*{4}}
   \put(30,5){\line(1,0){40}}
 \put(85,5){\makebox(0,0){$\ldots$}}
      \put(30,15){\makebox(0,0){$-c_{t''}$}}
        \put(60,15){\makebox(0,0){$-c_{t''-1}$}}
 \put(110,15){\makebox(0,0){$-c_{1}$}}

\put(130,5){\makebox(0,0){$\longrightarrow $}}

 \put(150,5){\circle*{4}}
    \put(180,5){\circle*{4}}
 \put(230,5){\line(-1,0){10}}
     \put(230,5){\circle*{4}}
   \put(150,5){\line(1,0){40}}
 \put(205,5){\makebox(0,0){$\ldots$}}
      \put(150,15){\makebox(0,0){$-c_{t''}$}}
        \put(180,15){\makebox(0,0){$-c_{{t''}-1}$}}
 \put(230,15){\makebox(0,0){$-c_{1}$}}


 \put(270,5){\circle*{4}}
    \put(310,5){\circle*{4}}  \put(360,5){\circle*{4}}
   \put(230,5){\line(1,0){90}}
    \put(360,5){\line(-1,0){10}}
 \put(335,5){\makebox(0,0){$\ldots$}}
      \put(270,15){\makebox(0,0){$-c_0-k_1$}}
        \put(310,15){\makebox(0,0){$-k_2$}}
          \put(360,15){\makebox(0,0){$-k_s$}}
 
\end{picture}
 
\noindent where $x_r/y_r=[[k_1,k_2,\ldots, k_s]]$ with 
$k_1\in \Z$, and $k_i\geq 2$ for $i\geq 2$. 

The graph at the right hand side is the $r'$ leg of the surgery manifold (after 
plumbing calculus which  transforms  it into 
a negative definite string, if it is needed). 

In particular, its determinant is $r'$ 
up to a sign $\epsilon _r=\pm 1$. If this string is negative definite, then the sign is $\epsilon _r =1$. If not, then 
in the procedure of the calculus we have to apply either a $(+1)$ blow down, or a 0-chain absorption, or a splitting which induces an orientation reversing (for oriented plumbing calculus see 
\cite{Neumann_plumbing}, or \cite[3.3.A]{Book}). According to this, if the string is not negative definite, we define  $\epsilon_r=-1$. 

Next, we compute the determinant of this string (the next formula can be obtained e.g. by the determinant formula from \cite[2.2.5(6)]{Book}, 
where $G$ plays the role of the edge
connecting the vertices decorated by 
$-c_1$ and $-c_0-k_1$). We obtain 
$\epsilon _r r'= r(c_0y_r+x_r)-\omega'_3y_r$.
But, from (\ref{eq:omega'}), $pq=c_0r-\omega'_3$, hence 
the two identities combined give
\begin{equation}\label{eq:r'}
\epsilon_r r'=x_r r +y_r pq.
\end{equation}
We have similar equations for the $q$ leg as well (whenever the coefficient is not $\infty$). 


We can also understand this in terms of the Dehn surgery description of a Brieskorn sphere. To retrieve~\eqref{eq:r'}, we can express the slope of the Seifert longitude in the final surgered manifold in terms of the Seifert numbers of $\Sigma(p, q, r)$ and the surgery coefficients $-x_q / y_q, -x_r / y_r$. Following \cite{EisenbudNeumann}, we can write the Seifert fibered space $\Sigma = \Sigma(p, q, r)$ as the union
\[  
    \Sigma = \Sigma_0 \cup (D^2 \times S^1)_1 \cup (D^2\times S^1)_2 \cup (D^2 \times S^1)_3,
\]
where $\Sigma_0$ is a non-singular $S^1$ bundle over $S^2 \smallsetminus (D^2_1 \cup D^2_2 \cup D_3^2)$. Let $Q$ be a section of $\Sigma_0$, with $Q_i = Q \cap \partial (D^2 \times S^1)_i$ (we use $Q_1$ and $Q_p$, etc. interchangeably). Let $F$ denote a generic non-singular fiber of $\Sigma_0$, and let $M_q, L_q, M_r, L_r$ denote the canonical meridian and longitude of the $q-$ and $r-$ singular fibers, respectively. 

As above, we compute the case for the $r$-fiber; the $q$-fiber runs along the same lines. In homology, $(M_r, L_r)$ and $(Q_r, F)$ are related as follows:
\begin{equation}
\begin{pmatrix}
    M_r \\
    L_r
\end{pmatrix}
= \left( \begin{matrix} r & \omega_3 \\ pq & - ( b pq+ \omega_1 q + \omega_2 p) \end{matrix} \right)
\begin{pmatrix}
    Q_r \\
    F
\end{pmatrix}.
\end{equation}
Note that this differs from the matrix from~\cite{EisenbudNeumann}, because they work in the unnormalized case, whereas we do not assume the Seifert number is zero, and because we want our Brieskorn spheres to be oriented so as to be negative definite. In the surgered manifold, we have $-x_r M_r + y_r L_r = 0$ on homology, which translates to
\[
\epsilon_r (pq y_r + r x_r ) Q_r +\epsilon_r (bpq y_r+ \omega_1 q y_r  + \omega_2 p y_r + \omega_3 x_r) F =0.
\]
so the new framing of the Seifert longitude for the $r$-fiber is precisely as in~\eqref{eq:r'}. The $q$-fiber case runs along the same lines, and the matrix form is as previously computed.

Any such surgery provides an integral homology sphere if and only if 
(see e.g.  \cite[Prop. 5.3.11]{Gompf1999}) 
\[\det 
\begin{pmatrix} x_q & py_q\\ py_r & x_r\end{pmatrix}
= \pm 1.\]

For the rest of the section, we calculate the surgery coefficients corresponding to our transformation matrices, $h_1, h_2$, and $h_4$. Throughout, we work with the case $p = 2$.

First, the transformation matrix $h_1$ gives 
$q'=q$ and $r'=r+2q$. Note $q'=q$ corresponds to the $\infty$-surgery, while for the $r$ leg we add a new vertex decorated by $-c_0-1$. Hence $x_r=y_r=1$, $\epsilon _r=1$, and as a result it is the $(-1)$-surgery along $K_r$.

For $h_4$, and for the $q$-leg, we add a 
vertex to the plumbing graph decorated by $-b_0+1$. Hence $x_q=-1$ and $y_q=1$. Since the string is negative definite $\epsilon_q=1$
(corresponding to the equation $q'=-q+2r$), this is a $+1$-surgery. 
Along the $r$ leg, we add a vertex decorated by $-c_0+3$. Since 
$c_0\leq 2$ this is a positive decoration (and in the plumbing calculus we use a $(+1)$ blow down). Moreover, $x_r=-3$, $y_r=1$, hence it is a $(+3)$-surgery and the equation is 
$-r'=-3r+2q$. Then the matrix $h_4 $ is obtained as 
\[\begin{pmatrix} \epsilon_q & 0\\0& \epsilon_r\end{pmatrix}\begin{pmatrix} x_q & py_q\\ py_r & x_r\end{pmatrix} = 
\begin{pmatrix} 1 & 0\\0& -1\end{pmatrix}\begin{pmatrix} -1 & 2\\ 2 & -3\end{pmatrix} = \begin{pmatrix} -1 & 2\\ -2 & 3\end{pmatrix}= h_4.
\]

In the case of the reflection transformation $S_2$, we have $S_2\binom{q}{r}=\binom{r'}{q'}$, hence the surgery operator here is in fact 
\[\begin{pmatrix} -5 & 4\\-4& 3\end{pmatrix}\begin{pmatrix} q  \\ r \end{pmatrix} = \begin{pmatrix} q' \\  r'\end{pmatrix}.\]
In this case, on the $q$ leg we glue two $(-2)$-vertices, and $b_0=4$, hence $k_1=-2$, $k_2=2$, that is, 
$x_q/y_q=-2-\frac{1}{2}=-5/2$. Hence this corresponds to a $(+5/2)$-surgery (and $\epsilon_q=1$).
On the other hand, on the $r$ leg, in order to `kill' the two $(-2)$-vertices we have to glue to the end a $(-1)$-vertex and to this one a $(-2)$-vertex. (Then, if we blow down these two consecutively, the old end $(-2)$-vertex transforms into a 0-vertex, and then we perform a `splitting' which eliminated these two end vertices.)
Hence $\epsilon _r=-1$, $c_0=2$, $-c_0-k_1=-1$, $k_1=-1$, 
$k_2=2$, $x_r/y_r=-1-\frac{1}{2}=-3/2$. It follows that this is a 
$(+3/2)$-surgery. 
The transformation matrix is:
\[\begin{pmatrix} \epsilon_q & 0\\0& \epsilon_r\end{pmatrix}\begin{pmatrix} x_q & py_q\\ py_r & x_r\end{pmatrix} = 
\begin{pmatrix} 1 & 0\\0& -1\end{pmatrix}\begin{pmatrix} -5 & 4\\ 4 & -3\end{pmatrix} = \begin{pmatrix} -5 & 4\\ -4 & 3\end{pmatrix}.
\]

\section{Certain lattice embeddings and positive \texorpdfstring{$d$-invariants }{d-invariants}}\label{s:6}

First we recall the following result of Lecuona and Lisca, a particular case of   \cite[Lemma 3.3]{LL}:
\begin{theorem} \label{th:LL}
Consider the following family of graphs $\Gamma_b$ for $b\geq 2$:

\begin{center}
\begin{tikzpicture}[
    v/.style={circle, fill, inner sep=0pt, minimum size=4pt},
    x=1.3cm, y=1.3cm
  ]
  \foreach \i/\w/\n in {0/-b/A, 1/-2/B, 3/-2/E, 4/-2/F}
    \node[v, label=above:{$\w$}, label=below:{$\n$}] (v\i) at (\i,0) {};

  \node[v, label=above:{$-2$}, label=below left:{$C$}] (v2) at (2,0) {};

  \node[v, label=left:{$D$}, label=right:{$-2$}] (d) at (2,-1) {};

  \draw (v0) -- (v1) -- (v2) -- (v3) -- (v4);
  \draw (v2) -- (d);

  \node[anchor=east] at (-0.4,0) {$\Gamma_b\colon$};
\end{tikzpicture}
\end{center}

Then the lattice $L_b=L(\Gamma_b)$ of $\Gamma_b$ cannot be embedded into a lattice of type $\langle -1\rangle ^{\oplus N}$. 
\end{theorem}
\begin{proof}
According to \cite{LL}, if the Seifert invariants 
$(-2; \{(\alpha_i, \omega_i)\}_{i=1}^3)$ satisfy
$1< \alpha_1/\omega_1\leq  \alpha_2/\omega_2\leq \alpha_3/\omega_3$
and $ 1< \omega _2/\alpha_2+\omega_3/\alpha_3$ then 
$L(-2; \{(\alpha_i, \omega_i)\}_{i=1}^3)$ cannot be embedded in any $\langle -1\rangle ^{\oplus r}$. In our case
$1< 3/2\leq  (2b-1)/b< 2/1$ and $1< b/(2b-1)+1/2$.
\end{proof}
If we change the canonical basis 
of the lattice 
$L(\Gamma_b)$ (provided by the plumbing presentation)
of the above graph $\Gamma_b$
by a $\Z$-linear isomorphism 
$(A,B,C,D,E,F)\mapsto (A, B-D, D,C,E,F)$,
 we get the following new  presentation of the same lattice $L_b$:

\begin{center}
\begin{tikzpicture}[
    v/.style={circle, fill, inner sep=0pt, minimum size=4pt},
    x=1.3cm
  ]
  \foreach \i/\w/\n in {0/-b/A, 1/-4/B-D, 2/-2/D, 3/-2/C, 4/-2/E, 5/-2/F}
    \node[v, label=above:{$\w$}, label=below:{$\n$}] (v\i) at (\i,0) {};
  \draw (v0) -- (v1)  (v2) -- (v3) -- (v4) -- (v5);
  \draw[double, double distance=2pt] (v1) -- (v2);

  \node[anchor=east] at (-0.4,0) {$\widetilde{\Gamma}_b\colon$};
\end{tikzpicture}
\end{center}

Here the double edge means intersection multiplicity 2. 

\begin{theorem}\label{th:po}
If the lattice $L_b=L(\widetilde{\Gamma}_b)$ (with $b\geq 2$)
can be embedded into the lattice $L(\Sigma(2,q,r))$ of 
$\Sigma(2,q,r)$, then $d(\Sigma(2,q,r))>0$.
\end{theorem}
\begin{proof}
    Combine \cref{th:LL} with a theorem of Elkies \cite{Elkies}, which states (reformulated for the case of  $\Sigma(2,q,r)$) that $L(\Sigma(2,q,r))=\langle -1\rangle ^{\oplus s}$ if and only if $d(\Sigma(2,q,r))=0$.
\end{proof}

\begin{theorem}\label{thm:nonvanishing}
    \[d(\Sigma(2,q,r))>0 \quad \text{ for all } q,r \text{ with } \quad \frac12<\frac{q}{r}<\frac23.
    \]
\end{theorem}
\begin{proof}
First we fix $q/r$ from the interval $(1/2, 3/5]$. Then a computation shows that $2r/q=[[b_0,b_1, \ldots]]$
has the form $[[4, b, \ldots]]$ for some $b\geq 2$. 
Furthermore, 
 $2q/r=[[c_0,c_1, \ldots]]=[[2,2,2,2,2,c_5,\dots ]]$.
Then the minimal good resolution graph of $\Sigma(2,q,r)$ has the following shape 
(with $b\geq 2$):
 
\begin{center}
\begin{tikzpicture}[
    v/.style={circle, fill, inner sep=0pt, minimum size=4pt},
    dots/.style={fill=white, inner sep=1pt, sloped},
    on grid,
    node distance=1.3cm
  ]
  \node[v, label=above:{$-2$}]                                 (a1) {};
  \node[v, right=of a1]                                        (a2) {};
  \node[v, right=of a2, label=above:{$-2$}]                    (r1) {};
  \node[v, right=of r1, label=above:{$-2$}, label=below:{$F$}] (r2) {};
  \node[v, right=of r2, label=above:{$-2$}, label=below:{$E$}] (r3) {};
  \node[v, right=of r3, label=above:{$-2$}, label=below:{$C$}] (r4) {};
  \node[v, below=of a2, label=left:{$A$}, label=right:{$-b$}]  (b)  {};

  \draw (a1) -- (a2);
  \draw (r1) -- (r2) -- (r3) -- (r4);
  \draw (a2) -- node[dots] {$\cdots$} (r1);
  \draw (a2) -- node[dots] {$\cdots$} (b);
\end{tikzpicture}
\end{center}

Next, we consider the following elements  in the lattice 
$L(\Sigma (2,q,r))$: $F, E, C, -C^*, A^*, -A$.
We claim that their intersection pattern is the following:

\begin{center}
\begin{tikzpicture}[
    v/.style={circle, fill, inner sep=0pt, minimum size=4pt},
    x=1.3cm
  ]
  \foreach \i/\w/\n in {0/-b/-A, 1/-4/A^*, 2/-2/-C^*, 3/-2/C, 4/-2/E, 5/-2/F}
    \node[v, label=above:{$\w$}, label=below:{$\n$}] (v\i) at (\i,0) {};
  \draw (v0) -- (v1)  (v2) -- (v3) -- (v4) -- (v5);
  \draw[double, double distance=2pt] (v1) -- (v2);
\end{tikzpicture}
\end{center}

The non-immediate part follows similarly as the identities (i)-(ii)-(iii) 
from Proposition \ref{ex:5}, namely 
$-(C^*,C^*)=c_0=2$, $-(A^*,A^*)=b_0=4$, $-(C^*,A^*)=2$ (the determinant of the remaining graph after we delete the path connecting the vertices of $A$ and $C$).

Now we can apply Theorem \ref{th:po}, and obtain that 
  $d(\Sigma(2,q,r))>0$ whenever  $1/2 < q/r \leq 3/5$.
Finally, we observe that the operator $S_2=(-4,3;-5,4)$ preserves the $d$-invariant for $q/r$ in the interval $(1/2, 2/3)$, cf.\! \cref{prop:EP}. But this operator sends 
 the interval $(1/2,3/5]$ into $[3/5,2/3)$. In particular, $d(\Sigma(2,q,r))>0$ whenever  $3/5 \leq q/r < 2/3$ as well.
\end{proof}

\section{Proof of Theorems 1, 2 and 3}\label{sec:proof_main_thms}

Let $\mathcal{M}$ be defined as in (\ref{eq:M}) and $\mathcal{V}$ as in Theorem \ref{th:int1} (II). First note that Theorem \ref{th:int2} and Theorem \ref{th:int3} 
are equivalent  by the commuting relations (\ref{eq:com2})
which identify the $\widetilde{\mathfrak{m}}$ action/orbits with the $\widetilde{\mathfrak{m}}^*$ action/orbits.

From the combinatorics of the Stern--Brocot sequences and from the main constructions  of the black/white intervals and V-fractions  we get that $(0,1)\cap \mathbb{Q}=\mathcal{M}\sqcup\mathcal{V}$.
Lemma \ref{lem:SB2} guarantees that the maximal even intervals are exactly the newly created intervals, and Lemma \ref{lem:SB4}
provides the correspondence between the newly created intervals and the V-fractions. This proves Theorem \ref{th:int1} without the corresponding $d$ values. Moreover, we also obtain that the union of $\widetilde{\mathfrak{m}}^*$ orbits of the elements of 
$[\frac{1}{2}, \frac{2}{3}]\cap \mathbb{Q}$ equals $\mathcal{M}$.

As a consequence, what we have to verify is that $h_1,h_2,h_4$
preserve the $d$ invariant  and that $d>0$ in the interval 
$[\frac{1}{2}, \frac{2}{3}]$. The first part follows from 
the fact that $h_1,h_2,h_4$ preserve the stable equivalence classes of the lattices (see \cref{s:Gs}). The second part 
is proved in \cref{s:6}. 

\section{Even lattices}\label{sec:even_plumbings}

The aim of this section is to give motivation for the calculations and methods of this paper. We do this in two ways, first, we summarize and reformulate the results of Yamada and Matsumoto from~\cite{Yamada_Matsumoto_1983} on so-called \emph{even plumbing} (defined below). Second, we calculate the value of $d$ for two specific families of Brieskorn spheres, illustrating the techniques of this paper in concrete setting.

\subsection{Even plumbings} In this section, we interpret the main theorem of Yamada and Matsumoto \cite{Yamada_Matsumoto_1983}, in the language of monoids in $\GL(2,\Z)$. They classified all Brieskorn spheres $\Sigma(2,q,r)$ which satisfy that the minimal resolution graph has all decorations even. We call them \textit{even plumbings}.

The $d$-invariant satisfies $s \geq 4d$,  where $s$ is the number of vertices of the minimal resolution graph. The even plumbing condition is equivalent to the relation $s = 4d$. Using this together with our results, we can deduce the result of Neumann and Yamada--Matsumoto as a corollary:

\begin{corollary}[\cite{Neumann_invariant,Yamada_Matsumoto_1983}]
    If $\Sigma(2,q,r)$ has even plumbing, then 
    $1/2 < q/r <2/3$.
\end{corollary}

\begin{proof}   
    We know under the action of the monoid generated by $h_1,h_2,h_4$ the unit interval $0<q/r<1$ splits into the orbits of elements of $[\frac{1}{2},\frac{2}{3}]$ and to the orbit of $1/3$. For the orbit of $1/3$, $d$ vanishes and clearly $s>4d$ in that case. Other fractions either lie in the interval $[\frac{1}{2},\frac{2}{3}]$ or they are images of one of the generators $h_1,h_2,h_4$. However, all three preserve $d$, while raising the number of vertices $s$ by $1$ or $2$, so no plumbing in the image can be even. 
\end{proof}

Consider the $\mathrm{GL}(2,\Z)$ matrices 
\[
    P = \Mat{9}{-4}{16}{-7},
\quad
T = \Mat{7}{4}{12}{7},\quad
Z = TU = \Mat{4}{7}{7}{12}.
\]
and 
\[\overline{P} = S_2 P, \quad \overline{T} = S_2 T,\quad \overline{Z} = S_2 Z,
\]
where $S_2 = \begin{inlinematrix}
    -4 & 3\\ -5 & 4
\end{inlinematrix}$.

We represent the plumbing graph of $\Sigma(2,q,r)$ by an array with two rows, describing (negatives of) the decorations on the $q$-leg and $r$-leg, for example $\Sigma(2,3,5)$ is represented as
\[
\ym{2,2}{2,2,2,2}.
\]
In a similar way as in the previous sections, one can check the following facts about the effect of the operators $P, T, Z$ on plumbing graphs for a given $q/r$ with $1/2<q/r<2/3$. 

\begin{enumerate}
\item
The operator $P$ sends
\(
\ym{x, 2k}{y}  
\)
to
\(
\ym{x, 2(k+1)}{y, 2, 2, 2, 2, 2, 2, 2, 2}, 
\)
\item 
The operator $T$ sends $\ym{x}{y}$
    to
    \(
    \ym{x, 8, 2}{y, 4, 2, 2, 2, 2, 2}
    \),
\item
The operator $Z$ sends
    \(
    \ym{x}{y}
    \)
    to
    \(
    \ym{y, 6, 2}{x, 6, 2, 2, 2, 2, 2}.
    \)
\end{enumerate}
Here $x$, and $y$ are arbitrary sequences and $k$ an arbitrary integer.
 
In the terminology of \cite{Yamada_Matsumoto_1983}, these 3 cases correspond to adding a \emph{particle} ($P^n$ adds the particle $p_n$), a \emph{$T$-joint}, and a \emph{$Z$-joint}, respectively. 

Consider now the submonoid $m_{ev}$ of $\GL(2,\Z)$ generated by $P,\overline{P},T,\overline{T},Z,\overline{Z}$. This monoid acts on pairs $(q,r) \in \Z^2$ as usual by the left multiplication of column vectors. 
\begin{theorem}[\cite{Yamada_Matsumoto_1983}] \leavevmode
    \begin{enumerate}
        \item The monoid $m_{ev}$ is free.
        \item The map $x \mapsto x \cdot 
\binom{3}{5}$ induces a bijection between $m_{ev}$ and the set of $\Sigma(2,q,r)$ with even plumbing. 
    \end{enumerate}
\end{theorem}

Although we merely reformulate the theorem of Yamada and Matsumoto, the operator viewpoint is rather enlightening. For example it makes clear that for all even plumbings, $q/r$ is in the union of (rather short) intervals
\begin{equation}\label{eq:intervals}
\left(\frac12,\frac{6}{11}\right) \cup \left(\frac{15}{26},  \frac{26}{45}\right) \cup  \left(\frac{29}{50},  \frac{18}{31}\right) \cup \Big\{\frac{3}{5}\Big\} \cup \left(\frac{21}{34},\frac{34}{55}\right) \cup \left(\frac{31}{50},\frac{18}{29}\right)\cup \left(\frac{9}{14},\frac{2}{3}\right),
\end{equation}
given by the images of $\left(\frac12,\frac23\right)$ under the generators of $m_{ev}$.

It is interesting to see which vertices of the Stern--Brocot tree correspond to even plumbings. For every such fraction in the sequence $\SB_n$, the operators $P$, and $\overline{P}$ create two new fractions in $\SB_{n+4}$ and the operators $T,\overline{T},Z,\overline{Z}$ create four new fractions in $\SB_{n+6}$. In particular, all even plumbings are in some $\SB_{2\ell+1}$ with $\ell \geq 1$.

\section{Examples}\label{sec:examples} 

It this section, we illustrate the main results of the paper on an explicit family of Brieskorn spheres. Namely, we reprove the following result:
\begin{corollary}\label{example1}
    For $q, n \in \mathbb{N}$ with $q$ odd, $q \geq 3$ and $n \geq 1$, 
    \begin{equation}\label{example2}
        d(\Sigma(2, q, 2q n - 1)) > 0
    \end{equation}
    and
    \begin{equation}\label{example3}
        d(\Sigma(2, q, 2qn + 1) ) = 0.
    \end{equation}
\end{corollary}
\noindent While the result given above is by no means new (see for instance~\cite{Tweedy_2013}),
the proof below provides a valuable illustration of the calculus of the Stern--Brocot tree from \cref{sec:sb-combinatorics}. We hope it will provide a model for anyone who may be interested in using our techniques to calculate $d(\Sigma(2, q, r))$ for any particular pair $q, r$ which they may be studying.
\begin{proof}
    First, we verify~\eqref{example2}. By \cref{th:int3} it is enough to show that $\frac{q}{2qn -1}$ is in the interior of an even interval. Suppose $\frac{q}{2qn - 1}$ first appears in $\mathrm{SB}_k$ for some $k$. We want to find $a, b, c, d \in \mathbb{N}$ such that the sequence
    \begin{equation}\label{example4}
    \frac{a}{b} < \frac{q}{2q n -1} < \frac{c}{d}
    \end{equation}
    appears in $\mathrm{SB}_k$. Recall that every rational number between 0 and 1 appears \emph{exactly once} in the Stern--Brocot tree, and that any interval $[a/b, c/d] \subseteq [0, 1]$ with $ad - bc = -1$ is an elementary interval in $\mathrm{SB}_\ell$ for \emph{some} $\ell$. Using the conditions that $a, b, c, d$ from~\eqref{example4} are integers, and:
    \begin{align*}
        ad- bc &= -1, \\
        a + c &= q, \\
        b + d & = 2qn - 1, \\
        a(2qn - 1) - b q &= -1,
    \end{align*}
    we get that
    \[
        \frac{a}{b} = \frac{1}{2n} \quad \text{ and } \quad \frac{c}{d} = \frac{q - 1}{2n (q - 1) - 1}.
    \]
    Repeating this process, we get the following subtree of the Stern--Brocot tree:
    \begin{center}
\begin{tikzpicture}[diagram,
    reset cm, x=8mm, y=10mm,                      
    every label/.append style={inner sep=2pt},
    omitted/.style={thick, dotted}]               

  \node[vertex, label={above left:$\frac{1}{2n}$}]           (a) at (0,0)  {};
  \node[vertex, label={above right:$\frac{2}{2\cdot 2n-1}$}] (b) at (4,-1) {};
  \node[vertex, label={above left:$\frac{3}{2\cdot 3n-1}$}]  (c) at (3,-2) {};
  \node[vertex, label={above left:$\frac{q-1}{2(q-1)n-1}$}]  (d) at (2,-3) {};
  \node[vertex, label={above left:$\frac{q}{2qn-1}$}]        (e) at (1,-4) {};

  \draw (a) -- (b) -- (c);
  \draw (d) -- (e);

      \draw[omitted] (a) -- ++(1.25,0.25);      
  \draw[omitted] (a) -- ++(-0.7,-0.7);    
  \draw[omitted] (b) -- ++(0.7,-0.7);     
  \draw[omitted] (c) -- (d);              
  \draw[omitted] (c) -- ++(0.7,-0.7);
  \draw[omitted] (d) -- ++(0.7,-0.7);
  \draw[omitted] (e) -- ++(-0.7,-0.7);
  \draw[omitted] (e) -- ++(0.7,-0.7);
\end{tikzpicture}
\end{center}
    Notice that $1/2n$ sits in the elementary interval $[\frac{0}{1}, \frac{1}{2n-1}]$, so it is in $\mathrm{SB}_{2n}$ and $[\frac{1}{2n}, \frac{2}{4n - 1}]$ is a (maximal) even interval. Clearly, this interval contains $\frac{q}{2qn - 1}$, so~\eqref{example2} is done.

    The proof of~\eqref{example3} works similarly. Letting $k$ temporarily denote the minimal integer with $\frac{q}{2q n + 1} \in \mathrm{SB}_k$, we start by finding the sequence  
    \[  
        \frac{a}{b} < \frac{q}{2q n +1} < \frac{c}{d}
    \]
    appearing in $\mathrm{SB}_k$ . We use conditions analogous to those above to obtain
    \[
        \frac{a}{b} = \frac{ q - 1}{2 (q - 1) n + 1} \quad \text{ and } \quad \frac{c}{d} = \frac{1}{2n}.
    \]
    Working backwards, we get a subtree of the Stern--Brocot tree of the form:
    \begin{center}
\begin{tikzpicture}[diagram,
    reset cm, x=8mm, y=10mm,                      
    every label/.append style={inner sep=2pt},
    omitted/.style={thick, dotted}]               

  \node[vertex, label={above left:$\frac{1}{2n}$}]            (a) at (0,0)   {};
  \node[vertex, label={above left:$\frac{1}{2n+1}$}]          (b) at (-4,-1) {};
  \node[vertex, label={above right:$\frac{2}{2\cdot 2n+1}$}]  (c) at (-3,-2) {};
  \node[vertex, label={above right:$\frac{q-1}{2(q-1)n+1}$}]  (d) at (-2,-3) {};
  \node[vertex, label={above right:$\frac{q}{2qn+1}$}]        (e) at (-1,-4) {};

  \draw (a) -- (b) -- (c);
  \draw (d) -- (e);

  \draw[omitted] (a) -- ++(1.25,0.25);      
  \draw[omitted] (a) -- ++(0.7,-0.7);     
  \draw[omitted] (b) -- ++(-0.7,-0.7);    
  \draw[omitted] (c) -- ++(-0.7,-0.7);
  \draw[omitted] (c) -- (d);              
  \draw[omitted] (d) -- ++(-0.7,-0.7);
  \draw[omitted] (e) -- ++(-0.7,-0.7);
  \draw[omitted] (e) -- ++(0.7,-0.7);
\end{tikzpicture}
\end{center}
    Any black interval containing $\frac{q}{2q n + 1}$ would have to come from above it in the tree. In fact the elementary intervals containing $\frac{q}{2q n + 1}$ in the interior have to be one of  
    \[
        \left[\frac{\alpha}{2 n \alpha + 1}, \frac{1}{2n }\right] \quad \text{with $1 \leq \alpha < q$, } \quad \text{ or } \quad \left[\frac{0}{1}, \frac{1}{m}\right]  \quad \text{for $1 \leq m \leq 2n$.}
    \]
    But since none of these intervals are of the form $E$ (counted mod 2) from Lemma~\ref{lem:SB1}, none of these intervals are even. Therefore, by Theorem~\ref{th:int3}, we obtain $q / (2qn + 1) \in \mathcal{V}$ for each odd $q \geq 3$ and $n \geq 1$, and~\eqref{example3} follows.
\end{proof}

\bibliographystyle{unsrt}
\bibliography{refs}

\end{document}

%% file: graphics/type-i-ii/interval_subdivision.tex
\begin{tikzpicture}[x=0.9cm, y=1cm,
      dot/.style    = {circle, fill, inner sep=0pt, minimum size=3.6pt},
      hollow/.style = {circle, draw, fill=white, inner sep=0pt,
                       minimum size=4.4pt, line width=0.6pt},
      lab/.style    = {above=3pt, inner sep=1pt},
      subint/.style = {line width=1.8pt}]

    \node[left] at (-1.15,0)    {$\SB_{2\ell}$};
    \node[left] at (-1.15,-1.5) {$\SB_{2\ell+2}$};

    \foreach \s in {0,7.6} {
      \begin{scope}[shift={(\s,0)}]
        \foreach \y in {0,-1.5} {
          \draw (-0.4,\y) -- (4.4,\y);
          \foreach \e in {0.6,0.8,1} {
            \fill (-\e,\y)  circle (0.6pt);
            \fill (4+\e,\y) circle (0.6pt);
          }
        }
        \draw[|->] (2,-0.32) -- (2,-0.72);
      \end{scope}
    }

    \foreach \x/\p/\q in {0/0/1,4/1/1} {
      \node[dot] at (\x,0) {};
      \node[lab] at (\x,0) {$\frac{\p}{\q}$};
    }
    \draw[subint] (2,-1.5) -- (3,-1.5);
    \foreach \x in {0,2,3,4} \node[dot] at (\x,-1.5) {};
    \node[hollow] at (1,-1.5) {};
    \foreach \p/\q [count=\x from 0] in {0/1,1/1,1/0,0/1,1/1}
      \node[lab] at (\x,-1.5) {$\frac{\p}{\q}$};
    \node at (2,-2.05) {Type I};

    \begin{scope}[shift={(7.6,0)}]
      \foreach \x/\p/\q in {0/1/1,4/1/0} {
        \node[dot] at (\x,0) {};
        \node[lab] at (\x,0) {$\frac{\p}{\q}$};
      }
      \draw[subint] (1,-1.5) -- (2,-1.5);
      \foreach \x in {0,1,2,4} \node[dot] at (\x,-1.5) {};
      \node[hollow] at (3,-1.5) {};
      \foreach \p/\q [count=\x from 0] in {1/1,1/0,0/1,1/1,1/0}
        \node[lab] at (\x,-1.5) {$\frac{\p}{\q}$};
      \node at (2,-2.05) {Type II};
    \end{scope}
\end{tikzpicture}

%% file: graphics/type-i-ii/sb_tree_type_i_ii.tex
\begin{tikzpicture}[diagram,
    reset cm, x=5mm, y=5mm,                     
    every label/.append style={inner sep=2pt}]  


  \begin{scope}[local bounding box=rowI]

    \node[vertex,  label={above:$\frac{0}{1}$}] (a1) at (0,-0.5) {};
    \node[circled, label={right:$\frac{1}{1}$}] (b1) at (3,0.5)  {};
    \draw[dashed] (a1) -- (b1);

    \draw[|->] (4.5,0) -- (5.7,0);

    \begin{scope}[shift={(6.2,0.5)}]
      \node[vertex,  label={above:$\frac{0}{1}$}]       (e1) at (0,0)     {};
      \node[circled, label={right:$\frac{1}{1}$}]       (t1) at (3,1)     {};
      \node[vertex,  label={above:$\frac{1}{0}$}]       (m1) at (1.5,-1)  {};
      \node[hollow,  label={above left:$\frac{1}{1}$}]  (h1) at (0.75,-2) {};
      \node[vertex,  label={above right:$\frac{0}{1}$}] (r1) at (2.25,-2) {};
      \draw[dashed] (t1) -- (e1);
      \draw (e1) -- (m1) -- (h1);
      \draw[bold edge] (m1) -- (r1);
    \end{scope}

    \begin{scope}[shift={(12.1,0)}]
      \node[vertex,  label={left:$\frac{0}{1}$}]  (a2) at (0,0.5)  {};
      \node[circled, label={above:$\frac{1}{1}$}] (b2) at (3,-0.5) {};
      \draw[dashed] (a2) -- (b2);

      \draw[|->] (4.5,0) -- (5.7,0);

      \begin{scope}[shift={(6.2,0.5)}]
        \node[vertex,  label={left:$\frac{0}{1}$}]        (t2) at (0,1)     {};
        \node[circled, label={above:$\frac{1}{1}$}]       (e2) at (3,0)     {};
        \node[vertex,  label={above:$\frac{1}{0}$}]       (m2) at (1.5,-1)  {};
        \node[hollow,  label={above left:$\frac{1}{1}$}]  (h2) at (0.75,-2) {};
        \node[vertex,  label={above right:$\frac{0}{1}$}] (r2) at (2.25,-2) {};
        \draw[dashed] (t2) -- (e2);
        \draw (e2) -- (m2) -- (h2);
        \draw[bold edge] (m2) -- (r2);
      \end{scope}
    \end{scope}
  \end{scope}

  \begin{scope}[shift={(0,-4.5)}, local bounding box=rowII]

    \node[circled, label={above:$\frac{1}{1}$}] (a3) at (0,-0.5) {};
    \node[vertex,  label={right:$\frac{1}{0}$}] (b3) at (3,0.5)  {};
    \draw[dashed] (a3) -- (b3);

    \draw[|->] (4.5,0) -- (5.7,0);

    \begin{scope}[shift={(6.2,0.5)}]
      \node[circled, label={above:$\frac{1}{1}$}]       (e3) at (0,0)     {};
      \node[vertex,  label={right:$\frac{1}{0}$}]       (t3) at (3,1)     {};
      \node[vertex,  label={above:$\frac{0}{1}$}]       (m3) at (1.5,-1)  {};
      \node[vertex,  label={above left:$\frac{1}{0}$}]  (l3) at (0.75,-2) {};
      \node[hollow,  label={above right:$\frac{1}{1}$}] (h3) at (2.25,-2) {};
      \draw[dashed] (t3) -- (e3);
      \draw (e3) -- (m3) -- (h3);
      \draw[bold edge] (m3) -- (l3);
    \end{scope}

    \begin{scope}[shift={(12.1,0)}]
      \node[circled, label={left:$\frac{1}{1}$}]  (a4) at (0,0.5)  {};
      \node[vertex,  label={above:$\frac{1}{0}$}] (b4) at (3,-0.5) {};
      \draw[dashed] (a4) -- (b4);

      \draw[|->] (4.5,0) -- (5.7,0);

      \begin{scope}[shift={(6.2,0.5)}]
        \node[circled, label={left:$\frac{1}{1}$}]        (t4) at (0,1)     {};
        \node[vertex,  label={above:$\frac{1}{0}$}]       (e4) at (3,0)     {};
        \node[vertex,  label={above:$\frac{0}{1}$}]       (m4) at (1.5,-1)  {};
        \node[vertex,  label={above left:$\frac{1}{0}$}]  (l4) at (0.75,-2) {};
        \node[hollow,  label={above right:$\frac{1}{1}$}] (h4) at (2.25,-2) {};
        \draw[dashed] (t4) -- (e4);
        \draw (e4) -- (m4) -- (h4);
        \draw[bold edge] (m4) -- (l4);
      \end{scope}
    \end{scope}
  \end{scope}

  \node[anchor=east] at ([xshift=-1em]rowI.west)  {Type I};
  \node[anchor=east] at ([xshift=-1em]rowII.west) {Type II};

  \path (rowI.south) -- (rowII.north) coordinate[pos=0.5] (hmid);
  \draw[thin, black!40] (current bounding box.west |- hmid)
                     -- (current bounding box.east |- hmid);
  \path (t1.center) -- (a2.center) coordinate[pos=0.5] (vmid);
  \draw[thin, black!40] (vmid |- current bounding box.north)
                     -- (vmid |- current bounding box.south);

\end{tikzpicture}

%% file: graphics/type-i-ii/sb_characterization_two.tex
\scalebox{0.9}{
\begin{tikzpicture}[
    x=0.97\linewidth,                       
    font=\small,
    seg/.style    = {line width=0.8pt},
    hl/.style     = {line width=2.4pt},     
    grayed/.style = {color=gray},       
    dot/.style    = {circle, fill, inner sep=0pt, minimum size=4.5pt},
    hollow/.style = {circle, draw, fill=white, line width=0.8pt,
                     inner sep=0pt, minimum size=6pt},
    frac/.style   = {above=5pt, inner sep=0pt},
    lab/.style    = {below=5pt, inner sep=0pt},
  ]

  \draw[seg] (0,0) -- (1,0);
  \node[lab] at (1/2,0) {$O$};
  \foreach \k/\p/\q in {0/0/1, 1/1/1} {
    \node[dot]  at (\k,0) {};
    \node[frac] at (\k,0) {$\dfrac{\p}{\q}$};
  }

  \begin{scope}[yshift=-2cm]
    \draw[seg] (0,0) -- (1,0);
    \foreach \k in {3}                      
      \draw[hl] ({(\k-1)/4},0) -- (\k/4,0);
    \foreach \k/\w in {1/g_1, 2/g_2u, 3/g_3, 4/g_4}
    \node[lab] at ({(\k-0.5)/4},0) {$O\w$};
    \foreach \k/\p/\q in {0/0/1, 1/1/3, 2/1/2, 3/2/3, 4/1/1} {
      \node[dot]  at (\k/4,0) {};
      \node[frac] at (\k/4,0) {$\dfrac{\p}{\q}$};
    }
    \foreach \k in {1}                      
      \node[hollow] at (\k/4,0) {};
  \end{scope}

  \begin{scope}[yshift=-4cm]
    \draw[seg]         (0,0)     -- (8/16,0);
    \draw[seg, grayed] (8/16,0)  -- (12/16,0);   
    \draw[seg]         (12/16,0) -- (1,0);
    \foreach \k in {3,6,15}                 
      \draw[hl] ({(\k-1)/16},0) -- (\k/16,0);
    \foreach \k/\w in { 1/g_1g_1,   2/g_1g_2u,   3/g_1g_3,   4/g_1g_4,
                        5/g_2g_4u,  6/g_2g_3u,  7/g_2g_2,   8/g_2g_1u,
                        9/g_3g_1,  10/g_3g_2u,  11/g_3g_3,  12/g_3g_4,
                       13/g_4g_1,  14/g_4g_2u, 15/g_4g_3,  16/g_4g_4}
        \node[lab] at ({(\k-0.5)/16}, {-0.45*mod(\k+1,2)}) {$O\w$};
        \foreach \k/\p/\q in { 0/0/1,  1/1/5,  2/1/4,  3/2/7,  4/1/3,  5/3/8,
                           6/2/5,  7/3/7,  8/1/2, 12/2/3, 13/5/7, 14/3/4,
                          15/4/5, 16/1/1}
      \node[frac] at (\k/16,0) {$\dfrac{\p}{\q}$};
    \foreach \k/\p/\q in {9/4/7, 10/3/5, 11/5/8}   
      \node[frac, grayed] at (\k/16,0) {$\dfrac{\p}{\q}$};
    \foreach \k in {0,1,2,3,4,5,6,7,8,12,13,14,15,16}
      \node[dot] at (\k/16,0) {};
    \foreach \k in {9,10,11}                
      \node[dot, grayed] at (\k/16,0) {};
    \foreach \k in {1,4,7,13}               
      \node[hollow] at (\k/16,0) {};
  \end{scope}

\end{tikzpicture}
}

%% file: graphics/plumbing/brieskorn_plumbing.tex
\begin{tikzpicture}[
    x=1cm, y=0.7cm,
    vertex/.style={circle, fill, inner sep=0pt, minimum size=4pt},
    every label/.style={font=\small, inner sep=2pt}
  ]
  \node[vertex, label={above:$b$}] (O) at (0,0) {};

  \node[vertex, label={above:$-a_t$}] (At) at (-1,0) {};
  \node[vertex, label={above:$-a_2$}] (A2) at (-2,0) {};
  \node[vertex, label={above:$-a_1$}] (A1) at (-3,0) {};

  \node[vertex, label={above:$-c_{t''}$}] (Ct) at (1,0) {};
  \node[vertex, label={above:$-c_2$}]     (C2) at (2,0) {};
  \node[vertex, label={above:$-c_1$}]     (C1) at (3,0) {};

  \node[vertex, label={right:$-b_{t'}$}] (Bt) at (0,-1) {};
  \node[vertex, label={right:$-b_2$}]    (B2) at (0,-2) {};
  \node[vertex, label={right:$-b_1$}]    (B1) at (0,-3) {};

  \draw (O) -- (At)  (O) -- (Bt)  (O) -- (Ct)
        (A2) -- (A1)  (B2) -- (B1)  (C2) -- (C1);
  \draw[thick, dotted] (At) -- (A2)  (Bt) -- (B2)  (Ct) -- (C2);
\end{tikzpicture}

%% file: graphics/plumbing/example_three.tex

\tikzset{
  vtx/.style = {circle, fill, inner sep=0pt, outer sep=0pt, minimum size=5pt},
  pv/.style 2 args = {
    vtx,
    label={[text depth=.5ex]above:$#1$},   
    label=below:$#2$
  },
  restbox/.style = {
    rectangle, draw=black!60, dashed, thin, fill=black!4,
    minimum width=2.2cm, minimum height=1.4cm
  },
  mapstoarrow/.style = {
    line width=0.6pt, {Bar[width=7pt]}-{To[length=3.5pt, width=7pt]}
  }
}

\begin{tikzpicture}[thick]

  \node[restbox]     (G1) at (0,0)     {$\Gamma - v$};
  \node[vtx]         (w1) at (G1.east) {};   
  \node[pv={e_v}{v}] (v1) at (2.3,0)   {};
  \draw (w1) -- (v1);
  \node at (0.7,-1.3) {$\Gamma$};

  \draw[mapstoarrow] (3.2,0) -- (4.4,0);

  \begin{scope}[xshift=6.2cm]
    \node[restbox]                    (G2)   at (0,0)     {$\Gamma - v$};
    \node[vtx]                        (w2)   at (G2.east) {};
    \node[pv={e_v}{v}]                (v2)   at (2.3,0)   {};
    \node[pv={-\Delta-1}{v_{\mathrm{new}}}] (vnew) at (3.5,0)   {};
    \draw (w2) -- (v2) -- (vnew);
    \node at (1.4,-1.3) {$\Gamma'$};
  \end{scope}

\end{tikzpicture}

%% file: graphics/plumbing/plumbing_before.tex
    \begin{tikzpicture}[
    v/.style={circle, fill, draw, inner sep=1.5pt},
    dots/.style={fill=white, inner sep=1pt, sloped},
  ]
  \node[v] (c)  at (0,0)  {};   
  \node[v] (p1) at (-1,0) {};   
  \node[v] (q1) at (0,-1) {};   
  \node[v] (n)  at (0,-2) {};   
  \node[v] (r1) at (1,0)  {};   
  \node[v] (r2) at (2,0)  {};
  \node[v] (i)  at (3,0)  {};   
  \node[v] (j)  at (4,0)  {};   

  \node[left]  at (n) {$E_n$};
  \node[below] at (i) {$E_i$};
  \node[below] at (j) {$E_j$};
  \node[above] at (i) {$-2$};
  \node[above] at (j) {$-2$};

  \draw (c)  -- node[dots, pos=0.6] {$\dots$} (p1);
  \draw (c)  -- node[dots, pos=0.6] {$\dots$} (q1);
  \draw (q1) -- (n);
  \draw (c)  -- (r1);
  \draw (r1) -- node[dots] {$\dots$} (r2);
  \draw (r2) -- (i) -- (j);
\end{tikzpicture}



%% file: graphics/plumbing/splice_max.tex
	\begin{tikzpicture}[
    v/.style={circle, fill, draw, inner sep=1.5pt},
    wt/.style={inner sep=1.5pt},                    
    dots/.style={fill=white, inner sep=1pt, sloped},
  ]
  \node[v] (c)  at (0,0)  {};   
  \node[v] (p1) at (-1,0) {};   
  \node[v] (q1) at (0,-1) {};   
  \node[v] (n)  at (0,-2) {};   
  \node[v] (r1) at (1,0)  {};   
  \node[v] (r2) at (2,0)  {};
  \node[v] (i)  at (3,0)  {};   
  \node[v] (j)  at (4,0)  {};   

  \node[left]  at (n) {$E_n$};
  \node[below] at (i) {$E_i$};
  \node[below] at (j) {$E_j$};

  \draw (c)  -- node[wt, pos=0.25, above] {$p$} node[dots, pos=0.6] {$\dots$} (p1);
  \draw (c)  -- node[wt, pos=0.25, right] {$q$} node[dots, pos=0.6] {$\dots$} (q1);
  \draw (n)  -- node[wt, pos=0.25, right] {$x$} (q1);
  \draw (c)  -- node[wt, pos=0.25, above] {$r$} (r1);
  \draw (r1) -- node[dots] {$\dots$} (r2);
  \draw (r2) -- node[wt, pos=0.75, above] {$y$} (i);
  \draw (i)  -- node[wt, pos=0.25, above] {$2$} node[wt, pos=0.75, above] {$z$} (j);
\end{tikzpicture}



%% file: graphics/plumbing/plumbing_after.tex
\begin{tikzpicture}[
    v/.style={circle, fill, draw, inner sep=1.5pt},
    dots/.style={fill=white, inner sep=1pt, sloped},
  ]
  \node[v] (c)  at (0,0)  {};   
  \node[v] (p1) at (-1,0) {};   
  \node[v] (q1) at (1,0)  {};
  \node[v] (n)  at (2,0)  {};   
  \node[v] (f)  at (3,0)  {};   
  \node[v] (j)  at (4,0)  {};   
  \node[v] (r1) at (0,-1) {};
  \node[v] (r2) at (0,-2) {};

  \node[above] at (n) {$E_n$};
  \node[above] at (f) {$F$};
  \node[above] at (j) {$E_j$};
  \node[below] at (f) {$-2$};
  \node[below] at (j) {$-2$};

  \draw (c)  -- node[dots, pos=0.6] {$\dots$} (p1);
  \draw (c)  -- node[dots, pos=0.6] {$\dots$} (q1);
  \draw (q1) -- (n) -- (f) -- (j);
  \draw (c)  -- (r1);
  \draw (r1) -- node[dots] {$\dots$} (r2);
\end{tikzpicture}
